\documentclass[a4paper,12pt]{amsart}
\usepackage[latin1]{inputenc}
\usepackage[english]{babel}
\usepackage{amsmath, amsthm, amssymb, amsopn, amsfonts, amstext, stmaryrd, enumerate, color, mathtools, hyperref, mathrsfs, accents}
\usepackage[margin=0.8in]{geometry}
\numberwithin{equation}{section}

\hypersetup{
	colorlinks=true,
	linkcolor=blue,
	citecolor=red,
	urlcolor=blue,
	linktoc=all}

\newcommand{\C}{\mathscr{C}}
\newcommand{\E}{\mathscr{E}}
\newcommand{\G}{\mathscr{G}}

\newcommand{\K}{\mathscr{K}}
\renewcommand{\L}{\mathscr{L}}
\newcommand{\LL}{\mathcal{L}}
\newcommand{\N}{\mathbb{N}}
\renewcommand{\P}{\mathscr{P}}
\newcommand{\Q}{\mathbb{Q}}
\newcommand{\R}{\mathbb{R}}
\renewcommand{\S}{\mathcal{S}}

\newcommand{\X}{\mathcal{X}}
\newcommand{\Z}{\mathbb{Z}}
\newcommand{\loc}{{\rm loc}}

\newcommand{\dist}{{\mbox{\normalfont dist}}}
\newcommand{\bequ}{\begin{equation}}
\newcommand{\nequ}{\end{equation}}
\newcommand{\PV}{\mbox{\normalfont P.V.}}
\newcommand{\Haus}{\mathcal{H}}
\newcommand{\QQ}{\mathcal{Q}}
\newcommand{\RR}{\mathcal{R}}
\newcommand{\YY}{\mathcal{Y}}
\newcommand{\GG}{\mathcal{G}}
\newcommand{\PP}{\mathcal{P}}
\newcommand{\tRn}{{\widetilde{\R}^n}}
\renewcommand{\u}{{u_\omega^M}}
\newcommand{\F}{{\mathscr{F}_\omega}}
\newcommand{\A}{{\mathcal{A}_\omega^M}}
\newcommand{\M}{{\mathcal{M}_\omega^M}}

\DeclareMathOperator{\card}{card}
\DeclareMathOperator{\Per}{Per}

\theoremstyle{plain}
\newtheorem{definition}{Definition}[section]
\newtheorem{theorem}[definition]{Theorem}
\newtheorem{proposition}[definition]{Proposition}
\newtheorem{lemma}[definition]{Lemma}
\newtheorem{corollary}[definition]{Corollary}

\theoremstyle{definition}

\renewcommand{\le}{\leqslant}
\renewcommand{\leq}{\leqslant}
\renewcommand{\ge}{\geqslant}
\renewcommand{\geq}{\geqslant}

\begin{document}

\title[Planelike minimizers of nonlocal energies and fractional perimeters]{Planelike minimizers of nonlocal Ginzburg-Landau energies and fractional perimeters in periodic media}

\author[Matteo Cozzi, Enrico Valdinoci]{
Matteo Cozzi${}^{(1,2)}$
\and
Enrico Valdinoci${}^{(3,4)}$
}



\thanks{The first author is supported by the MINECO grants~MDM-2014-0445 and~MTM2017-84214-C2-1-P. The second author is supported by the Australian Research Council grant~N.E.W. ``Nonlocal Equations at Work''.}

\maketitle

{\scriptsize \begin{center} (1) -- BGSMath Barcelona Graduate School of Mathematics.
\end{center}
\scriptsize \begin{center}(2) -- Departament de Matem\`atiques\\
Universitat Polit\`ecnica de Catalunya\\
Diagonal 647, E-08028 Barcelona (Spain).
\end{center}
\scriptsize \begin{center} (3) --
School of Mathematics and Statistics\\
University of Melbourne\\
Grattan Street,
Parkville, VIC-3010 Melbourne
(Australia).
\end{center}
\scriptsize \begin{center} (4) -- Dipartimento di Matematica
``Federigo Enriques''\\
Universit\`a
degli Studi di Milano\\
Via Saldini 50, I-20133 Milano (Italy).
\end{center}
\bigskip

\begin{center}
E-mail addresses: matteo.cozzi@upc.edu,
enrico@math.utexas.edu
\end{center}
}

\begin{abstract}
We consider here a nonlocal phase transition energy in a periodic medium
and we construct solutions whose interfaces lie at a bounded distance
from any given hyperplane.

These solutions are either periodic or quasiperiodic, depending
on the rational dependency of the normal direction to the reference
hyperplane.

Remarkably, the oscillations of the interfaces with respect to
the reference
hyperplane are bounded by a universal constant times the periodicity
scale of the medium.

This geometric property allows us to
establish, in the limit,
the existence of planelike nonlocal
minimal surfaces in a periodic structure.

The proofs rely on new optimal density and energy estimates.
In particular, roughly speaking, the energy of phase transition
minimizers is controlled, both from above and below, by the energy
of one-dimensional transition layers.
\end{abstract}

\tableofcontents

\section{Introduction}\label{0orhteriteru6877}

In this paper, we consider a phase transition model in a periodic
medium with long-range particle interactions.
As customary, the phase coexistence is mathematically described
by a double-well potential, the minimization of which tends to
set a suitable state-parameter function into 
one of the two pure phases (which will be taken here to be~$-1$ and~$+1$).

In order to make the coexistence of phases significant
from both the mathematical and the
physical point of view, the total energy
of the system has to take into account also an elastic, or
ferromagnetic, energy, which avoids the production of
unnecessary phase changes
and forces the interface between phases to be minimal, at least
at large scales, with respect to a suitable notion of surface tension.

The model that we study here considers an elastic
energy of nonlocal type, that takes into account
long-range particle interactions with polynomial decay. {F}rom
the mathematical point of view, this elastic energy
takes the form of a suitable seminorm of Gagliardo
type which is related to fractional Sobolev spaces, see e.g.~\cite{SV12, SV14}.
For this, the nonlocal character of the energy is encoded
into a fractional parameter~$s\in(0,1)$, and the smaller this
parameter is, the stronger the nonlocal effect on the system.

This type of models also finds natural applications in the description
of boundary layer
effects on phase transitions, see~\cite{ABS94, AB98, G09, SiV12},
and in classical equations subject to a nonlinear boundary reaction,
see~\cite{CS-M05, SiV09, CC10, CC14, CS15}.
Also, for recent books on lonlocal problems see e.g.~\cite{MBRS16, BV17, DMV17, G17}.
\medskip

At a large scale, the phase separation tends to
minimize a suitable notion of surface tension, related to
either local or nonlocal perimeters, see~\cite{SV12}.
In this, the fractional parameter~$s=1/2$ provides
a threshold between local and nonlocal behaviors of interfaces
at a large scale: indeed, when~$s\in[1/2,\,1)$ these interfaces
are related to the minimization of a classical perimeter functional,
and the nonlocal effects are not involved in this process;
conversely, when~$s\in(0,\,1/2)$ these interfaces
are related to the minimization of the fractional perimeter functional
introduced in~\cite{CRS10} and thus
the nonlocal effects persist at any scale.
\medskip

The particular focus of this paper is on periodic media,
and for this we suppose that both the potential and the elastic
energies depend periodically on the space variable (of course,
a natural interpretation of such model comes from the study of crystals).
The periodicity scale of the medium is given by a parameter~$\tau>0$.
In crystals, one may think that~$\tau$ is small: for this, it is
natural to seek results which possess good scaling properties
with respect to~$\tau$.
\medskip

The main result of this paper is indeed the construction
of phase transition solutions of minimal type whose
interfaces lie at a bounded distance from any prescribed hyperplanes.
Roughly speaking, these interfaces separate
the pure states in an ``almost flat'' way.
We stress that the existence of these objects is not obvious, since
the medium is not homogeneous, and the flatness
property of the interfaces
is global in the whole of the space.\medskip

Moreover, and most importantly, such flatness property will be shown
to have optimal scaling features with respect to the periodicity size
of the medium. Namely, the distance from the prescribed
hyperplanes will be bounded by a structural constant times~$\tau$.
That is, in the motivation coming from crystallography,
the oscillation of these interfaces will be proven to be
comparable with the size of the crystal itself.\medskip

This invariance by scaling will allow us also to
look at a rescaled version of this picture and, by developing
an appropriate $\Gamma$-convergence theory in this framework,
we will also obtain a result on fractional minimal surfaces
in periodic media. Indeed, we will establish the
existence of nonlocal minimal surfaces in a periodic setting
which stay at a bounded
distance from any prescribed hyperplane (the same
result for classical minimal surfaces was obtained in~\cite{CdlL01}).\medskip

A crucial step in the proof of our results lies
in obtaining
density and energy estimates that are sharp with respect to
the size of the fundamental domain and that possess optimal scaling properties.
As a matter of fact, the energy of the minimizers
is not expected to behave in an additive way with respect to the domain
(roughly, the energy in a double ball is not the sum of the energy
of two balls). This is due to the fact that
the minimizers have the tendency to concentrate their energy
along a codimension-one interface. In addition, the nonlocal
features of the elastic energy contribute significantly to
the total
energy and this type of contribution changes dramatically
in dependence on the fractional parameter~$s$: once again,
the analysis for the cases~$s\in(0,\,1/2)$, $s=1/2$ and~$s\in
(1/2,\, 1)$ require different methods and give different results.

As a matter of fact, we will see that the energy
of the minimizers in a ball of radius~$R$ is controlled from above
and below by~$R^{n-\min\{1,2s\}}$ for~$s\in(0,1)\setminus\{1/2\}$,
and a logarithmic correction is needed for the case~$s=1/2$.\medskip

The precise mathematical setting in which we work is the following.
For a domain~$\Omega \subseteq \R$, we define
\begin{equation} \label{Edef}
\E(u; \Omega) := \frac{1}{2} \iint_{\C_{\Omega}} \left| u(x) - u(y) \right|^2 K(x, y) \, dx dy + \int_\Omega W(x, u(x)) \, dx,
\end{equation}
where
\begin{equation} \label{COmegadef}
\C_\Omega := \Big( \R^n \times \R^n \Big) \setminus \Big( \left( \R^n \setminus \Omega \right) \times \left( \R^n \setminus \Omega \right) \Big).
\end{equation}

The kernel~$K: \R^n \times \R^n \to [0, +\infty]$ is a measurable function satisfying
\begin{equation} \label{Ksymmetry} \tag{K1}
K(x, y) = K(y, x) \quad \mbox{for a.a.~} x, y \in \R^n,
\end{equation}
and
\begin{equation} \label{Kbounds} \tag{K2}
\frac{\lambda \chi_{(0, \xi)}(|x - y|)}{|x - y|^{n + 2 s}} \le K(x, y) \le \frac{\Lambda}{|x - y|^{n + 2 s}} \quad \mbox{for a.a.~} x, y \in \R^n,
\end{equation}
for some~$s \in (0, 1)$,~$\Lambda \ge \lambda > 0$ and~$\xi > 0$. Assumption~\eqref{Kbounds} ensures that~$K$ is controlled from above and below (in a neighborhood of the diagonal of~$\R^{2 n}$) by the standard homogeneous, translation invariant, rotationally symmetric kernel
\begin{equation} \label{FLker}
K_s(x, y) := \frac{1}{|x - y|^{n + 2 s}}.
\end{equation}
When~$s \in [1/2, 1)$, we also impose the regularity assumption
\begin{equation} \label{Kreg} \tag{K3}
\left| K(x, x + w) - K(x, x - w) \right| \le \Gamma |w|^{- n - 1 + \nu} \quad \mbox{for a.a.~} x, w \in \R^n,
\end{equation}
for some~$\nu \in (0, 1)$ and~$\Gamma > 0$.

On the other hand, the potential~$W: \R^n \times \R \to [0, +\infty)$ is a measurable function for which
\begin{equation} \label{Wzeros} \tag{W1}
W(x, \pm 1) = 0 \quad \mbox{for a.a.~} x \in \R^n,
\end{equation}
and, for any~$\theta \in [0, 1)$,
\begin{equation} \label{Wgamma} \tag{W2}
\inf_{\substack{x \in \R^n \\ |r| \le \theta}} W(x, r) \ge \gamma(\theta),
\end{equation}
where~$\gamma$ is a non-increasing positive function of the interval~$[0, 1)$. Moreover, we require~$W$ to be locally of class~$C^1$ in the second variable, to satisfy
\begin{equation} \label{Wbound} \tag{W3}
W(x, r), \, |W_r(x, r)| \le \kappa^{-1} \quad \mbox{for a.a.~} x \in \R^n \mbox{ and any } r \in \R,
\end{equation}
and
\begin{equation} \label{W''} \tag{W4}
\begin{aligned}
W(x, t) \ge W(x, r) + \kappa (1 + r) (t - r) + \kappa (t - r)^2 & \quad \mbox{for any } -1 \le r \le t \le - 1 + \kappa, \\
W(x, r) \ge W(x, t) + \kappa (1 - t) (t - r) + \kappa (t - r)^2 & \quad \mbox{for any } 1 - \kappa \le r \le t \le 1, \\
W(x, r) \le \kappa^{-1} (1 - |r|) & \quad \mbox{for any } 1 - \kappa \le |r| \le 1,
\end{aligned}
\end{equation}
for a.a.~$x \in \R^n$ and some small constant~$\kappa \in (0, 1)$. Condition~\eqref{W''} essentially tells that~$W$ must have superquadratic/sublinear detachment from its zeroes~$\pm 1$, uniformly in~$x$. We point out that any potential which is locally~$C^2$ in the second variable and satisfies
$$
W_r(x, \pm 1) = 0 \mbox{ and } W_{rr}(x, \pm 1) \ge \kappa \quad \mbox{for a.a.~} x \in \R^n,
$$
also fulfills~\eqref{W''}. But much more general behaviors are allowed.

In order to model a periodic environment, we need to impose periodicity conditions on the kernel~$K$ and the potential~$W$. Given~$\tau > 0$, we assume
\begin{equation} \label{Kper} \tag{K4}
K(x + k, y + k) = K(x, y) \quad \mbox{for a.a.~} x, y \in \R^n \mbox{ and any } k \in \tau \Z^n.
\end{equation}
\begin{equation} \label{Wper} \tag{W5}
W(x + k, r) = W(x, r) \quad \mbox{for a.a.~} x \in \R^n \mbox{ and any } k \in \tau \Z^n,
\end{equation}
for any fixed~$r \in \R$. We notice that the parameter~$\tau$ allows us to modulate the periodicity scale of the medium.
\medskip

Typical examples of potentials $W(x,r)$ that we take
into account are given by
\begin{eqnarray*}
&& Q(x)\, (1-r^2)^2,\\
&& Q(x)\, |1-r^2|^d \quad{\mbox{ with }} d\in (1,2),\\
&& Q(x)\, \left(1+\cos(\pi r)\right)\\
{\mbox{and }}&&
Q(x)\,\cos^2\left(\frac{\pi r}{2}\right),
\end{eqnarray*}
with~$Q(x)\in[1,2]$ for any~$x\in\R^n$.

Typical examples of interaction kernels are given by
$$
K(x, y) = \frac{a(x, y)}{|x - y|^{n + 2 s}},
$$
with~$a$ Lipschitz continuous and periodic in both~$x$ and~$y$.\medskip

Sometimes we will adopt shorthand notations for the interaction and potential terms appearing in definition~\eqref{Edef}. Given any measurable sets~$E, F \subseteq \R^n$, we write
\begin{equation} \label{kindef}
\begin{aligned}
\K_K(u; E, F) & = \K(u; E, F) := \frac{1}{2} \int_E \int_F \left| u(x) - u(y) \right|^2 K(x, y) \, dx dy \\
\K_K(u; E) & = \K(u; E) := \frac{1}{2} \iint_{\C_E} \left| u(x) - u(y) \right|^2 K(x, y) \, dx dy,
\end{aligned}
\end{equation}
with~$\C_E$ as in~\eqref{COmegadef}, and
$$
\P_W(u; E) = \P(u; E) := \int_E W(x, u(x)) \, dx.
$$
Under these conventions (and the symmetry assumption~\eqref{Ksymmetry}), we have
\begin{align*}
\E(u; \Omega) & = \K(u; \Omega, \Omega) + 2 \K(u; \Omega, \R^n \setminus \Omega) + \P(u; \Omega) \\
& = \K(u; \Omega) + \P(u; \Omega).
\end{align*}

\medskip

Often, we will consider the integral operator~$\L_K$ associated with~$\E$, that is defined by
\begin{equation} \label{LKdef}
\begin{aligned}
\L_K u(x) := & \, \PV \int_{\R^n} \left( u(x) - u(y) \right) K(x, y) \, dy \\
= & \, \lim_{\varepsilon \rightarrow 0^+} \int_{\R^n \setminus B_\varepsilon(x)} \left( u(x) - u(y) \right) K(x, y) \, dy.
\end{aligned}
\end{equation}
Such operator naturally appears when considering the Euler-Lagrange equation of the functional~$\E$. We observe that~$\L_K u$ is well-defined pointwise, at least when~$u$ is a smooth bounded function and~$K$ satisfies~\eqref{Kreg}, when~$s \ge 1/2$. See Lemmata~\ref{LKlems<} and~\ref{LKlemsge} for estimates in this direction. We also notice that~$\L_K$ boils down to the well-known fractional Laplacian~$(-\Delta)^s$, when~$K$ is the standard kernel~$K_s$ as defined in~\eqref{FLker}.

\medskip

After these preliminary definitions, we are now almost ready to state the main results contained in this paper. In order to do this, we first need to make precise the notions of minimizers of the functional~$\E$ that we take into consideration.

\begin{definition} \label{mindef}
Let~$\Omega$ be an open subset of~$\R^n$. A measurable function~$u: \R^n \to \R$ is said to be a~\emph{minimizer} of~$\E$ in~$\Omega$ if~$\E(u; \Omega) < +\infty$ and
$$
\E(u; \Omega) \le \E(v; \Omega),
$$
for any measurable function~$v$ that coincides with~$u$ outside of~$\Omega$.
\end{definition}

This definition may be extended to the whole space in the following way.

\begin{definition} \label{classAmindef}
A measurable function~$u: \R^n \to \R$ is said to be a~\emph{class~A minimizer} of~$\E$ if it is a minimizer for~$\E$ in every bounded open set~$\Omega \subset \R^n$.
\end{definition}

We stress that the simpler notion of~\emph{global minimizer} is too restrictive for our scopes. Indeed, the functions that we typically take into consideration have infinite energy over the whole space~$\R^n$ and therefore it is convenient to evaluate their energy on bounded domains only. However, class~A minimizers of~$\E$ are still (weak) solutions of the Euler-Lagrange equation
$$
- 2 \, \L_K u = W_u(\cdot, u) \quad \mbox{in } \R^n.
$$

The concept of class~A minimizers is frequently used in the literature, in contexts where one has to deal with objects having only locally finite energy. The terminology comes from~\cite{M24} and has been more recently adopted in e.g.~\cite{CdlL01,V04,PV05,CV17,CP16}.

\medskip

In what follows, we will construct class~A minimizers of~$\E$ and related functionals, which exhibit a close-to-one-dimensional geometry.

More specifically, we will look for minimizers that connect the two pure phases~$-1$ and~$1$ of the potential~$W$ asymptotically in one fixed direction~$\omega \in \R^n \setminus \{ 0 \}$ of the space, and that are~\emph{planelike}, in the sense that their intermediate level sets (of levels between, say,~$-9/10$ and~$9/10$) are contained in a strip orthogonal to~$\omega$ and of width universally proportional to the periodicity scale~$\tau$ of the medium. Note that 
when~$s<1/2$
we shall call \emph{universal} any
quantity that depends at most on~$n$,~$s$,~$\lambda$,~$\Lambda$,~$
\kappa$, the function~$\gamma$, but not on~$\xi$ and~$\tau$.
A similar notation is taken when~$s\ge 1/2$, but in this case
universal quantities may also depend on~$\nu$ and~$\Gamma$, 
according to condition~\eqref{Kreg}.
Namely, when~$s\ge1/2$,
we shall call \emph{universal} any
quantity that depends at most on~$n$,~$s$,~$\lambda$,~$\Lambda$,~$
\kappa$, the function~$\gamma$, 
$\nu$, $\Gamma$
but not on~$\xi$ and~$\tau$.\medskip

Our construction will heavily rely on the periodic structure of the ambient space and will be carried out in different ways, depending on whether the direction~$\omega$ belongs to~$\tau \Q^n$ or not. In the first, \emph{rational} case, the minimizers will naturally inherit a periodic property from that of the medium, in a sense that may be made precise through the following definition.

\begin{definition} \label{simperdef}
Let~$\omega \in \tau \Q^n \setminus \{ 0 \}$. We define the equivalence relation~$\, \sim_{\tau, \, \omega}$ in~$\R^n$, by setting
$$
x \sim_{\tau, \, \omega} y \quad \mbox{if and only if} \quad x - y = k \in \tau \Z^n \mbox{ and } \omega \cdot k = 0.
$$
We say that a function~$u: \R^n \to \R$ is~\emph{periodic with respect to~$\sim_{\tau, \, \omega}$} or simply~\emph{$\sim_{\tau, \, \omega}$-periodic} if
$$
u(x) = u(y) \quad \mbox{for any } x, y \in \R^n \mbox{ such that } x \sim_{\tau, \, \omega} y.
$$
\end{definition}

When no confusion may arise, we will denote this equivalence relation simply with~$\sim$.

\medskip

We are now in position to present the statements of the main contributions of this paper.
Our first result (Theorem~\ref{tauPLthm})
improves the main theorem of~\cite{CV17}
and it allows its application to the
scaled energies that provides $\Gamma$-limit results as a byproduct
(see Theorems~\ref{epsPLthm} and~\ref{PerPLthm}).
In this sense, our main results here consist in the forthcoming
Theorem~\ref{tauPLthm} and in the sequence of arguments
leading from it to 
Theorems~\ref{epsPLthm} and~\ref{PerPLthm}.

We point out that
the result in Theorem~\ref{tauPLthm} is valid for the whole fractional parameter
range~$s\in(0,1)$, while the $\Gamma$-limit results focus on the
strongly nonlocal regime~$s\in(0,1/2)$, in which the nonlocal features
of the problem are preserved at any scale and produce, in the limit,
a nonlocal perimeter functional (we think that it will be also interesting
to investigate the $\Gamma$-limit in the weakly nonlocal regime~$s\in[1/2,1)$,
and in this framework a limit functional of local type has to be expected).

\begin{theorem} \label{tauPLthm}
Let~$n \ge 2$ and~$s \in (0, 1)$. Assume that the kernel~$K$ and the potential~$W$ respectively satisfy~\eqref{Ksymmetry},~\eqref{Kbounds},~\eqref{Kper} and~\eqref{Wzeros},~\eqref{Wgamma},~\eqref{Wbound},~\eqref{W''},~\eqref{Wper}, with~$\xi = \tau \ge 1$. If~$s \in [1/2, 1)$, we also require~$K$ to fulfill~\eqref{Kreg}.\\
For any fixed~$\theta \in (0, 1)$, there exists a constant~$M_0 > 0$, depending only on~$\theta$ and on universal quantities, such that, given any direction~$\omega \in \R^n \setminus \{ 0 \}$, we can construct a class~A minimizer~$u$ of the energy~$\E$ for which
\begin{equation} \label{tauPLcond}
\bigg\{ x \in \R^n : \left| u(x) \right| < \theta \bigg\} \subset \bigg\{ x \in \R^n : \frac{\omega}{|\omega|} \cdot x \in \left[ 0, \tau M_0 \right] \bigg\}.
\end{equation}
Furthermore,
\begin{enumerate}[$\bullet$]
\item if~$\omega \in \tau \Q^n \setminus \{ 0 \}$, then~$u$ is periodic with respect to~$\sim_{\tau, \, \omega}$, while
\item if~$\omega \in \R^n \setminus \tau \Q^n$, then~$u$ is the locally uniform limit of a sequence of periodic class~A minimizers of~$\E$.
\end{enumerate}
\end{theorem}

Theorem~\ref{tauPLthm} has been proved in~\cite{CV17} for the case in which the periodicity scale~$\tau$ is equal to~$1$. Its proof in the more general setting of this paper requires several important modifications from that of~\cite{CV17}. We shall comment more on the differences in the argument at the end of this section.

In local contexts, similar results have been obtained in~\cite{V04} and~\cite{CdlL01}, where the authors respectively took into account an energy with gradient interaction term and a geometric functional driven by a heterogeneous perimeter, instead the one appearing in~\eqref{Edef}. See also~\cite{PV05,PV05b,NV07,BV08,D13} for related constructions.

In nonlocal frameworks, Theorem~\ref{tauPLthm} here and~\cite[Theorem~1.4]{CV17} are the first available results on planelike minimizers, to the best of our knowledge. When the medium is homogeneous (i.e.~$K$ is translation invariant and~$W$ does not depend on~$x$), the existence of one-dimensional minimizers has been investigated in~\cite{CS-M05,PSV13,CS14,CS15,CP16}.

When comparing Theorem~\ref{tauPLthm} to~\cite[Theorem~1.4]{CV17}, it is worth noting that the presence of a medium with~$\tau$-periodicity is mostly reflected at the level of the minimizers in the fact that the constructed minimizers have level sets contained in a strip of width proportional to~$\tau$, as can be seen in~\eqref{tauPLcond}. Besides being interesting in itself (and not obtainable with the techniques of~\cite{CV17}), this fact leads to important consequences when applied to the class of scaled functionals that we now introduce.

\medskip

Given a small~$\varepsilon > 0$, we define the scaled energy~$\E_\varepsilon$ on any measurable set~$\Omega \subset \R^n$ as
\begin{equation} \label{Eepsdef}
\E_\varepsilon(u; \Omega) = \frac{1}{2} \iint_{\C_{\Omega}} \left| u(x) - u(y) \right|^2 K(x, y) \, dx dy + \frac{1}{\varepsilon^{2 s}} \int_\Omega W(x, u(x)) \, dx.
\end{equation}
This modified functional has been first studied in~\cite{SV12} for the model case of~$K$ given by~\eqref{FLker} and naturally arises when considering the rescaling
\begin{equation} \label{Reps}
\mathcal{R}_\varepsilon u(x) := u \left( \frac{x}{\varepsilon} \right).
\end{equation}

With the aid of~\eqref{Reps}, it is almost immediate to see that Theorem~\ref{tauPLthm} implies the following analogous result for the functional~$\E_\varepsilon$.

\begin{theorem} \label{epsPLthm}
Let~$n \ge 2$ and~$s \in (0, 1)$. Assume that the kernel~$K$ and the potential~$W$ respectively satisfy~\eqref{Ksymmetry},~\eqref{Kbounds},~\eqref{Kper} and~\eqref{Wzeros},~\eqref{Wgamma},~\eqref{Wbound},~\eqref{W''},~\eqref{Wper}, with~$\xi = \tau \ge 1$. If~$s \in [1/2, 1)$, we also require~$K$ to fulfill~\eqref{Kreg}.\\
For any fixed~$\theta \in (0, 1)$, there exists a constant~$M_0 > 0$, depending only on~$\theta$ and on universal quantities, such that, given any~$\varepsilon \in (0, \tau]$ and any direction~$\omega \in \R^n \setminus \{ 0 \}$, we can construct a family
of class~A minimizers~$u_\varepsilon$ of the energy~$\E_\varepsilon$ for which
$$
\bigg\{ x \in \R^n : \left| u_\varepsilon(x) \right| < \theta \bigg\} \subset \bigg\{ x \in \R^n : \frac{\omega}{|\omega|} \cdot x \in \left[ 0, \tau M_0 \right] \bigg\}.
$$
Furthermore,
\begin{enumerate}[$\bullet$]
\item if~$\omega \in \tau \Q^n \setminus \{ 0 \}$, then~$u_\varepsilon$ is periodic with respect to~$\sim_{\tau, \, \omega}$, while
\item if~$\omega \in \R^n \setminus \tau \Q^n$, then~$u_\varepsilon$ is the locally uniform limit of a sequence of periodic class~A minimizers of~$\E_\varepsilon$.
\end{enumerate}
\end{theorem}

Observe that the family of minimizers~$\{ u_\varepsilon \}$ produced by Theorem~\ref{epsPLthm} is such that each minimizer has intermediate values confined in a strip of width independent of~$\varepsilon$. For this being true, it is crucial that the value~$M_0$ found in Theorem~\ref{tauPLthm} does not depend on the periodicity scale~$\tau$. Such uniform-in-$\varepsilon$ width of the strips where the transitions of the~$u_\varepsilon$'s occur allows to consider smaller and smaller values of~$\varepsilon$ and eventually take the limit as~$\varepsilon \rightarrow 0^+$.

In the remaining part of this first section we shall focus on what happens when one takes this limit.

\medskip


In the classical Van der Waals-Cahn-Hilliard theory, a gradient term weighted by a small parameter~$\varepsilon$ is often introduced in the total energy functional in order to model phase coexistence phenomena that exhibit smooth transition interfaces. See e.g.~\cite{CGS84,G87} and the references therein for some more detailed explanations on the subject.

In~\cite{M87}, in particular, the limit as~$\varepsilon \rightarrow 0^+$ of these~$\varepsilon$-scaled functionals has been deeply analyzed through the language of~$\Gamma$-convergence. It has been proved there that the interfaces of the minimizers of such functionals converge to a minimal surface, building therefore a bridge between the Allen-Cahn-Ginzburg-Landau energy and the De Giorgi perimeter.

Nonlocal variants of this~$\Gamma$-convergence result have also been considered. Typically, one replaces the gradient penalization with a term that takes into account finite differences and allows for long-range interactions. In~\cite{ABS94,AB98,G09}, the authors obtained $\Gamma$-convergence results in which the target functional is still the classical perimeter, in conformity with the classical theory. More recently, a wider array of behaviors for the limit functional has been discovered in~\cite{SV12}. There, it is shown that (a suitable renormalization in $\varepsilon$ of) the family of energies~\eqref{Eepsdef}, with kernel~$K$ given by~\eqref{FLker},~$\Gamma$-converges to the standard perimeter when~$s\ge1/2$, and to the new notion of~\emph{fractional perimeter} introduced in~\cite{CRS10} when~$s < 1/2$.
\smallskip

In what follows, we study the~$\Gamma$-limit of the functional~$\E_\varepsilon$ in~\eqref{Eepsdef} in the strongly nonlocal regime~$s < 1/2$. We pose such restriction since we are predominantly interested in the emerging of a nonlocal perimeter of the type of~\cite{CRS10} and, in particular, in deducing from Theorem~\ref{epsPLthm} an analogous statement for the minimal surfaces of such perimeter.

We nevertheless believe that it might be interesting to investigate the~$\Gamma$-limit also in the case of~$s \ge 1/2$, presumably obtaining a local, heterogeneous perimeter.

\medskip

As said right above, we now restrict our attention to kernels that 
satisfy condition~\eqref{Kbounds} in Section~\ref{0orhteriteru6877} with~$s \in (0, 1/2)$.

Given an open set~$\Omega \subseteq \R^n$ and a measurable set~$E \subset \R^n$, we define the~\emph{$K$-perimeter} of~$E$ inside~$\Omega$ as
\begin{equation} \label{PerKdef}
\Per_K(E; \Omega) := \LL_K(E \cap \Omega, \Omega \setminus E) + \LL_K(E \cap \Omega, \R^n \setminus (E \cup \Omega)) + \LL_K(E \setminus \Omega, \Omega \setminus E),
\end{equation}
where, for any two disjoint measurable sets~$A, B \subset \R^n$, we set
$$
\LL_K(A, B) := \int_A \int_B K(x, y) \, dx dy.
$$

We stress that, when~$K$ is given by~\eqref{FLker}, the~$K$-perimeter boils down to the fractional perimeter introduced in~\cite{CRS10}.

Anisotropic versions of this nonlocal perimeter have first been studied in~\cite{L14}. The very recent paper~\cite{CSV16} deals with an even more general class of anisotropic perimeter functionals, driven by kernels which are not necessarily homogeneous. With definition~\eqref{PerKdef}, we consider a perimeter that may possibly be also~\emph{space-dependent} and therefore model a completely heterogeneous environment.

\smallskip

In the following, we study the minimizers of~$\Per_K$, especially in the whole space~$\R^n$. In analogy with Definitions~\ref{mindef} and~\ref{classAmindef}, we consider the following concepts of minimizers.

\begin{definition}
Let~$\Omega$ be an open subset of~$\R^n$. Given a measurable set~$E \subset \R^n$, we say that its boundary~$\partial E$ is a~\emph{minimal surface} for~$\Per_K$ in~$\Omega$ if~$\Per_K(E; \Omega) < +\infty$ and
$$
\Per_K(E; \Omega) \le \Per_K(F; \Omega),
$$
for any measurable set~$F$ such that~$F \setminus \Omega = E \setminus \Omega$.
\end{definition}

\begin{definition}
The boundary~$\partial E$ of a measurable set~$E \subset \R^n$ is said to be a~\emph{class A minimal surface} for~$\Per_K$ if it is a minimal surface for~$\Per_K$ in every bounded open set~$\Omega \subset \R^n$.
\end{definition}

\medskip

As a first result, we show that~$\Per_K$ is the~$\Gamma$-limit of the functionals~$\E_\varepsilon$ defined in~\eqref{Eepsdef}, in the appropriate topology. Observe that, in the notation of~\eqref{kindef}, we may write
\begin{equation} \label{PerKchi}
\Per_K(E; \Omega) = \frac{1}{4} \, \K_K(\chi_E - \chi_{\R^n \setminus E}; \Omega).
\end{equation}
We then introduce the space of functions
$$
\X := \bigg\{ u \in L^\infty(\R^n) : \| u \|_{L^\infty(\R^n)} \le 1 \bigg\},
$$
and we endow it with the~$L^1_\loc(\R^n)$ topology. In view of the representation
in~\eqref{PerKchi}, the~$K$-perimeter may be seen as acting on the subset of~$\X$ composed by the modified characteristic functions of the form~$\chi_E - \chi_{\R^n \setminus E}$, for measurable sets~$E$. Actually, we may extend it to a functional~$\G_K(\cdot, \Omega): \X \to [0, +\infty]$ by setting
$$
\G_K(u; \Omega) := \begin{cases}
\K_K(u; \Omega) & \quad \mbox{if } u|_\Omega = \chi_E - \chi_{\R^n \setminus E} \mbox{ for some measurable } E \subseteq \Omega\\
+\infty & \quad \mbox{otherwise}.
\end{cases}
$$
When no confusion may arise, we omit the dependence of~$\G_K$ on the kernel~$K$ and simply refer to this functional as~$\G$.

\smallskip

We have the following~$\Gamma$-convergence result.

\begin{proposition} \label{Gammaconvprop}
Let~$n \ge 1$ and~$s \in (0, 1/2)$. Assume the kernel~$K$ to be a non-negative function satisfying~\eqref{Ksymmetry} and the potential~$W$ to fulfill conditions~\eqref{Wzeros},~\eqref{Wgamma},~\eqref{Wbound}.\\
Then, the family
of functionals~$\E_\varepsilon$~$\Gamma$-converges to~$\G$ on~$\X$. That is,
\begin{enumerate}[$(i)$]
\item for any~$u_\varepsilon$ converging to~$u$ in~$\X$, it holds
$$
\G(u; \Omega) \le \liminf_{\varepsilon \rightarrow 0^+} \E_\varepsilon(u_\varepsilon; \Omega) \quad \mbox{for any open set } \Omega \subseteq \R^n;
$$
\item for any~$u \in \X$, there exists~$u_\varepsilon$ converging to~$u$ in~$\X$ such that
$$
\G(u; \Omega) \ge \limsup_{\varepsilon \rightarrow 0^+} \E_\varepsilon(u_\varepsilon; \Omega) \quad \mbox{for any open set } \Omega \subseteq \R^n.
$$
\end{enumerate}
\end{proposition}

In view of the above proposition, we see that the~$K$-perimeter is the correct geometric counterpart to the energy functional~$\E_\varepsilon$, when~$s < 1/2$. Consequently, the minimizers of~$\Per_K$ might be treated as limits of the minimizers of~$\E_\varepsilon$.

In this spirit, we may take the limit as~$\varepsilon \rightarrow 0^+$ in Theorem~\ref{epsPLthm} and obtain the existence of planelike minimal surfaces for the~$K$-perimeter. Before doing this, we extend Definition~\ref{simperdef} to get a notion of~$\sim_{\tau, \omega}$-periodicity for the subsets of~$\R^n$.

\begin{definition}
We say that a set~$A \subseteq \R^n$ is~\emph{periodic with respect to~$\sim_{\tau, \, \omega}$} or simply~\emph{$\sim_{\tau, \, \omega}$-periodic} if
$$
x \in A \mbox{ implies that } y \in A \mbox{ for any } y \in \R^n \mbox{ such that } y \sim_{\tau, \, \omega} x.
$$
\end{definition}

We are now in position to state the following

\begin{theorem} \label{PerPLthm}
Let~$n \ge 2$ and~$s \in (0, 1/2)$. Assume that the kernel~$K$ satisfies conditions~\eqref{Ksymmetry},~\eqref{Kbounds} and~\eqref{Kper}, with~$\xi = \tau \ge 1$.\\
There exists a universal constant~$M_0 > 0$ for which, given any direction~$\omega \in \R^n \setminus \{ 0 \}$, we can construct a class~A minimal surface~$\partial E$ for the perimeter~$\Per_K$ such that
$$
\bigg\{ x \in \R^n : \frac{\omega}{|\omega|} \cdot x < 0 \bigg\} \subset E \subset \bigg\{ x \in \R^n : \frac{\omega}{|\omega|} \cdot x \le \tau M_0 \bigg\}.
$$
Furthermore,
\begin{enumerate}[$\bullet$]
\item if~$\omega \in \tau \Q^n \setminus \{ 0 \}$, then~$\partial E$ is periodic with respect to~$\sim_{\tau, \, \omega}$, while
\item if~$\omega \in \R^n \setminus \tau \Q^n$, then~$\partial E$ is the locally uniform limit of a sequence of periodic class~A minimal surfaces for~$\Per_K$.
\end{enumerate}
\end{theorem}

Recently, we also obtained a version of Theorem~\ref{PerPLthm}
with different methods, by taking a suitable limit of an Ising model on a lattice,
see Theorem~1.7 of~\cite{CDV17b}.\medskip

In local frameworks, a counterpart of Theorem~\ref{PerPLthm} has first been obtained in~\cite{CdlL01} for a class of periodic perimeter functionals in~$\R^n$ and more general Riemannian contexts. This result generalizes to any dimension classical statements on the existence of geodesics on~$2$-dimensional periodic manifolds (\cite{M24,H32}). The result of~\cite{CdlL01} has then been obtained again in~\cite{V04} via a~$\Gamma$-convergence approach that motivates and inspires ours here.

\medskip

Before heading to the proofs of the results previously stated, we conclude this introductory section with a brief remark on the argument that we follow to prove our main contribution, Theorem~\ref{tauPLthm}.

The strategy adopted is, in some steps, close
to the one developed in~\cite{V04} and later translated to this nonlocal setting in~\cite{CV17}. Basically, we first restrict ourselves to directions~$\omega$ that belong to~$\tau \Q^n$, as the~\emph{irrational} ones may be dealt with a limiting procedure. For such~\emph{rational}~$\omega$'s, we use the compactness provided by the equivalence relation~$\sim$ to construct an appropriately constrained minimizer~$\u$ for~$\E$ in the strip
$$
\S_\omega^M = \Big\{ x \in \R^n : \omega \cdot x \in [0, M] \Big\},
$$
for any large~$M > 0$. The proof then finishes by proving that, for~$M$ large enough, the candidate~$\u$ is indeed a class~A minimizer, as desired.

\smallskip

The essential, and important,
difference between the proof provided here and that of~\cite{CV17} lies in the conclusive step.

In~\cite{CV17}, the proof that~$\u$ is a class~A minimizer relies on the coupling of uniform~$C^\alpha$ estimates on~$\u$ with suitable bounds from above on the growth of the energy~$\E$ on large balls. Such energy estimates have first been proved in~\cite{CC10,CC14,SV14} for functionals related to the fractional Laplacian. A more general version of this result can be found in~\cite{CV17}. Setting
\begin{equation} \label{Psidef}
\Psi_s(t) := \begin{cases}
t^{1 - 2 s} & \quad \mbox{if } s \in (0, 1/2) \\
\log t      & \quad \mbox{if } s = 1/2 \\
1           & \quad \mbox{if } s \in (1/2, 1),
\end{cases}	
\end{equation}
for~$t > 1$, it may be stated as follows.

\begin{proposition}[\cite{CV17}] \label{enestprop}
Let~$n \ge 1$,~$s \in (0, 1)$,~$x_0 \in \R^n$ and~$R \ge 3$. Assume that~$K$ and~$W$ satisfy~\eqref{Ksymmetry},~\eqref{Kbounds} and~\eqref{Wzeros},~\eqref{Wbound}, respectively. If~$u: \R^n \to [-1, 1]$ is a minimizer of~$\E$ in~$B_{R + 2}(x_0)$, then
\begin{equation} \label{enest}
\E(u; B_R(x_0)) \le C R^{n - 1} \Psi_s(R),
\end{equation}
for some constant~$C \ge 1$ which depends on~$n$,~$s$,~$\Lambda$ and~$\kappa$.
\end{proposition}

Though simple and rather elementary, the proof of~\cite{CV17} does not scale well with the parameter~$\tau$, in the sense that, once applied in the framework of this paper, it does not provide information
on the dependence on~$\tau$ of the width of the strip appearing on the right-hand side of inclusion~\eqref{tauPLcond}. As one can easily convince himself, the dependence stated in~\eqref{tauPLcond} (that is, width of the strip~$\simeq M_0 \tau$) is crucial for obtaining Theorem~\ref{epsPLthm}. Furthermore, the reason for which such dependence can not be detected is that the H\"older estimates of~\cite[Section~2]{CV17} do not rely at all on the 
double-well structure of the potential~$W$ and do not match appropriately with the rescaling~\eqref{Reps}.

Here, we solve this issue by correcting our strategy and making it more adherent to those traced in~\cite{CdlL01,V04}. In more concrete terms, we replace the use of the~$C^\alpha$ bounds with a powerful tool, frequently associated with Allen-Cahn equations and minimal surfaces: the density estimates.

\smallskip

Density estimates are a classical device in geometric measure theory, where they are used to study minimal surfaces. In PDEs, they have first been introduced in~\cite{CC95} to obtain the uniform convergence of the level sets of the minimizers to Ginzburg-Landau-type energies driven by the~$L^2$ norm of the gradient. Later on, they have been generalized to more general functionals with gradient structures (see for instance~\cite{V04,PV05,PV05b,NV07}) and, more recently, to nonlocal energies driven by Gagliardo seminorms (see~\cite{SV11,SV14}).

\smallskip

In this paper, we present a version of the density estimates compatible with our setting. To obtain such results, we suitably modify the arguments of~\cite{SV11,SV14}.

Moreover, we couple these density estimates with a bound from below on the energy of non-trivial minimizers of~$\E$, which is a counterpart of Proposition~\ref{enestprop} and a proof of their general optimality.

As we believe that these two results may be interesting in themselves and useful in other contexts, we include their statements here in the introduction.

\begin{theorem} \label{densestthm}
Let~$n \ge 2$ and~$s \in (0, 1)$. Assume~$K$ and~$W$ to respectively satisfy~\eqref{Ksymmetry},~\eqref{Kbounds} and~\eqref{Wzeros},~\eqref{Wgamma},~\eqref{Wbound},~\eqref{W''}. If~$s \in [1/2, 1)$, also suppose that~$K$ fulfills~\eqref{Kreg}.\\
Let~$u: \R^n \to [-1, 1]$ be a minimizer of~$\E$ in~$B_R(x_0)$, for some~$x_0 \in \R^n$ and~$R > 0$. Fix~$\theta, \theta_0 \in (-1, 1)$. If~$u(x_0) \ge \theta_0$, then there exist two constants~$\bar{c} \in (0, 1)$, depending on universal quantities, and~$\bar{R} = \bar{R}(\theta, \theta_0) \ge 2$, that may also depend on~$\theta$ and~$\theta_0$, such that
\begin{equation} \label{densest1}
\left| \left\{ u > \theta \right\} \cap B_R(x_0) \right| \ge \bar{c} R^n,
\end{equation}
provided~$\bar{R} \le R \le \xi/3$. Similarly, if~$u(x_0) \le \theta_0$, then
\begin{equation} \label{densest2}
\left| \left\{ u < \theta \right\} \cap B_R(x_0) \right| \ge \bar{c} R^n,
\end{equation}
provided~$\bar{R} \le R \le \xi/3$.
\end{theorem}

\begin{theorem} \label{enestbelowthm}
Let~$n \ge 2$ and~$s \in (0, 1)$. Assume~$K$ and~$W$ to respectively satisfy~\eqref{Ksymmetry},~\eqref{Kbounds} and~\eqref{Wzeros},~\eqref{Wgamma},~\eqref{Wbound},~\eqref{W''}. If~$s \in [1/2, 1)$, also suppose that~$K$ fulfills~\eqref{Kreg}.\\
Let~$u: \R^n \to [-1, 1]$ be a minimizer of~$\E$ in~$B_R(x_0)$, for some~$x_0 \in \R^n$ and~$R > 0$. If~$u(x_0) \in [-\theta_0, \theta_0]$, for some~$\theta_0 \in (0, 1)$, then there exist two constants~$c_0 \in (0, 1)$ and~$R_0 \ge 1$, depending only on~$\theta_0$ and on universal quantities, such that
\begin{equation} \label{enestbelow}
\E(u; B_R(x_0)) \ge c_0 R^{n - 1} \Psi_s(R),
\end{equation}
provided~$R_0 \le R \le \xi$.
\end{theorem}

See also Proposition~\ref{intdensprop} in Section~\ref{densec} for a bound from below on the measure of the interface of non-trivial minimizers. This result has been already
announced\footnote{
We take this opportunity to correct some imprecisions contained in~\cite{CDV17a}.
First of all, in formula~(3.3) the expressions ``$eithern$'' and ``$orn$''
should be replaced by~``either~$n$'' and ``or~$n$'', respectively.
More interestingly, formula~(2.2) has to be replaced by
\begin{gather*}
\left| \{ |u_\varepsilon|<\vartheta_2 \} \cap B_1 \right| \geq c \varepsilon
\\
\mbox{and, if } s\in(1/2, 1), \, \left| \{ |u_\varepsilon|<\vartheta_2\}\cap B_1 \right| \leq C \varepsilon.
\end{gather*}
Note that the lower bound is true for any~$s \in (0, 1)$ and can be easily deduced from Proposition~\ref{intdensprop} by scaling. On the other hand, the condition~$s\in(1/2, 1)$ for the upper bound slipped out of the original formula~(2.2) in~\cite{CDV17a}, and this in fact generates the very interesting
question on whether this upper bound may hold true also
when~$s\in (0, 1/2]$. To the best of our knowledge,
the only known upper bound in such generality is
$$
\left| \{ |u_\varepsilon|<\vartheta_2\}\cap B_1 \right| \leq C \begin{cases}
\varepsilon^{2 s} & \quad \mbox{if } s \in (0, 1/2) \\
\varepsilon |\log \varepsilon| & \quad \mbox{if } s = 1/2,
\end{cases}
$$
which can be obtained from an appropriate rescaling
of the energy estimate of Proposition~\ref{enestprop}, namely Proposition~\ref{epsenestprop} in Section~\ref{PerPLsec}.

Actually, as a result of Theorem 1.1(v) of the recent preprint~\cite{MSW16},
the better bound~$ | \{ |u_\varepsilon| < \vartheta_2 \} \cap B_1 |
\le C_\alpha \varepsilon^{\min \{ 4 s, \alpha \}}$,
for every~$\alpha \in (0, 1)$, holds true when~$ s \in (0, 1/2)$,
for a special class of minimizers and, more in general,
solutions of an Allen-Cahn equation driven by the
fractional Laplace operator of order~$2s$. This result was obtained in~\cite{MSW16}
by means of extension methods which cannot be applied directly to deal
with general integrodifferential kernels.
}
in~\cite{CDV17a} in an equivalent formulation for the minimizers~$u_\varepsilon$ of the rescaled functional~\eqref{Eepsdef}.
\medskip

The remaining part of this paper is organized as follows.

In Section~\ref{auxsec} we gather some technical and auxiliary results that will be used throughout the following sections.

Sections~\ref{densec} and~\ref{ensec} are respectively devoted to the proofs of the density estimates of Theorem~\ref{densestthm} and the energy estimate of Theorem~\ref{enestbelowthm}. An important consequence of the combination of these two results - the so-called~\emph{clean ball condition} - is contained in Proposition~\ref{cleanballprop}, at the end of Section~\ref{ensec}.

In Section~\ref{PLminsec} we address Theorem~\ref{tauPLthm}. As, for a large part, the proof follows closely the one displayed in~\cite[Sections~4 and~5]{CV17}, we only sketch the general argument and focus on the key differences.

Section~\ref{Gammasec} contains the proof of the~$\Gamma$-convergence result stated in Proposition~\ref{Gammaconvprop}.

In the conclusive Section~\ref{PerPLsec} we deal with Theorem~\ref{PerPLthm}, by showing how it can be deduced from Theorem~\ref{epsPLthm}, with the aid of some arguments closely related to Proposition~\ref{Gammaconvprop}.

\section{Some auxiliary results} \label{auxsec}

In this preliminary section we collect a few ancillary results of different nature, that will be needed in the remaining part of the paper.

We begin with a couple of lemmata, containing standard pointwise estimates on the integral operator~$\L_K$ defined in~\eqref{LKdef}. For~$s < 1/2$, we have

\begin{lemma} \label{LKlems<}
Let~$n \ge 1$ and~$s \in (0, 1/2)$. Assume that~$K$ satisfies~\eqref{Kbounds}. Then, given~$x \in \R^n$,~$\rho > 0$ and~$\psi \in L^\infty(\R^n) \cap C^{0, 1}(B_\rho(x))$, it holds
\begin{equation} \label{LKests<}
|\L_K \psi(x)| \le C \Lambda \left( \| \psi \|_{L^\infty(\R^n)} \rho^{- 2 s} + \| \nabla \psi \|_{L^\infty(B_\rho(x))} \rho^{1 - 2 s} \right),
\end{equation}
for some constant~$C > 0$ which depends on~$n$ and~$s$.
\end{lemma}
\begin{proof}
We split the integral defining~$\L_K \psi(x)$ as
$$
\L_K \psi(x) = I_1 + I_2,
$$
where
\begin{align*}
I_1 & := \int_{\R^n \setminus B_\rho(x)} \left( \psi(x) - \psi(y) \right) K(x, y) \, dy \\
I_2 & := \int_{B_\rho(x)} \left( \psi(x) - \psi(y) \right) K(x, y) \, dy.
\end{align*}
Applying~\eqref{Kbounds}, we first compute
$$
\left| I_1 \right| \le 2 \Lambda \| \psi \|_{L^\infty(\R^n)} \int_{\R^n \setminus B_\rho(x)} |x - y|^{- n - 2 s} \, dy = \frac{n |B_1| \Lambda \| \psi \|_{L^\infty(R^n)} \rho^{- 2 s}}{s}.
$$
To control~$I_2$ we use the Lipschitzianity of~$\psi$ together with~\eqref{Kbounds} again to get
$$
\left| I_2 \right| \le \Lambda \| \nabla \psi \|_{L^\infty(B_\rho(x))} \int_{B_\rho(x)} |x - y|^{1 - n - 2 s} \, dy = \frac{n |B_1| \Lambda \| \nabla \psi \|_{L^\infty(B_\rho(x))} \rho^{1 - 2 s}}{1 - 2 s}.
$$
These two estimates combined lead to~\eqref{LKests<}.
\end{proof}

In the general case of~$s \in (0, 1)$ - and, most significantly, when~$s \ge 1/2$ - a similar statement holds, as long as we add the regularity assumption~\eqref{Kreg} on the kernel~$K$.

\begin{lemma} \label{LKlemsge}
Let~$n \ge 1$ and~$s \in (0, 1)$. Assume that~$K$ satisfies~\eqref{Kbounds} and~\eqref{Kreg}. Then, given~$x \in \R^n$,~$\rho > 0$ and~$\psi \in L^\infty(\R^n) \cap C^{1, 1}(B_\rho(x))$, it holds
\begin{equation} \label{LKestsge}
|\L_K \psi(x)| \le C \left[ \Lambda \left( \| \psi \|_{L^\infty(\R^n)} \rho^{- 2 s} + \| \nabla^2 \psi \|_{L^\infty(B_\rho(x))} \rho^{2 (1 - s)} \right) + \Gamma |\nabla \psi(x)| \rho^\nu \right],
\end{equation}
for some constant~$C > 0$ which depends on~$n$,~$s$ and~$\nu$.
\end{lemma}
\begin{proof}
We have
\begin{align*}
|\L_K \psi(x)| & \le \int_{\R^n \setminus B_\rho(x)} |\psi(x) - \psi(y)| K(x, y) \, dy \\
& \quad + \int_{B_\rho(x)} |\psi(x) - \psi(y) + \nabla \psi(x) \cdot (y - x)| K(x, y) \, dy \\
& \quad + \left| \PV \int_{B_\rho(x)} \nabla \psi(x) \cdot (y - x) K(x, y) \, dy \right| \\
& =: I_1 + I_2 + I_3.
\end{align*}
The term~$I_1$ can be estimated exactly as in Lemma~\ref{LKlems<} to deduce
\begin{equation} \label{I1}
I_1 \le \frac{n |B_1| \Lambda \| \psi \|_{L^\infty(\R^n)} \rho^{- 2 s}}{s}.
\end{equation}
On the other hand, the regularity of~$\psi$ and again~\eqref{Kbounds} imply that
\begin{equation} \label{I2}
I_2 \le \Lambda \| \nabla^2 \psi \|_{L^\infty(B_\rho(x))} \int_{B_\rho(x)} |x - y|^{2 - n - 2 s} \, dy = \frac{n |B_1| \Lambda \| \nabla^2 \psi \|_{L^\infty(B_\rho(x))} \rho^{2 (1 - s)}}{2 (1 - s)}.
\end{equation}
Finally, we claim that
\begin{equation} \label{I3}
I_3 \le \frac{n |B_1| \Gamma |\nabla \psi(x)| \rho^\nu}{\nu}.
\end{equation}
The proof of~\eqref{I3} is a little bit more involved. Of course, we may assume that~$\nabla \psi(x) \ne 0$, as if the contrary is true, then~\eqref{I3} follows trivially. Write
$$
\PV \int_{B_\rho(x)} \nabla \psi(x) \cdot (y - x) K(x, y) \, dy = \lim_{\varepsilon \rightarrow 0^+} \int_{B_\rho(x) \setminus B_\varepsilon(x)} \nabla \psi(x) \cdot (y - x) K(x, y) \, dy,
$$
and consider the half-annuli
\begin{align*}
A_\varepsilon^+ & := \left\{ y \in B_\rho(x) \setminus B_\varepsilon(x) : \nabla \psi(x) \cdot (y - x) \ge 0 \right\} \\
A_\varepsilon^- & := \left( B_\rho(x) \setminus B_\varepsilon(x) \right) \setminus A_\varepsilon^+,
\end{align*}
for any small~$\varepsilon > 0$. Then, 
\begin{align*}
\int_{B_\rho(x) \setminus B_\varepsilon(x)} & \nabla \psi(x) \cdot (y - x) K(x, y) \, dy \\
& = \int_{A_\varepsilon^+} \nabla \psi(x) \cdot (y - x) K(x, y) \, dy + \int_{A_\varepsilon^-} \nabla \psi(x) \cdot (y - x)  K(x, y) \, dy \\
& = \int_{A_\varepsilon^+} \nabla \psi(x) \cdot (y - x) K(x, y) \, dy + \int_{A_\varepsilon^+} \nabla \psi(x) \cdot (x - z) K(x, 2 x - z) \, dz \\
& = \int_{A_\varepsilon^+} \nabla \psi(x) \cdot (y - x) \left[ K(x, y) - K(x, 2 x - y) \right] dy,
\end{align*}
where we applied the change of variables~$z := 2 x - y$ to the integral over~$A_\varepsilon^-$ and we noticed that this map is a diffeomorphism of~$A_\varepsilon^-$ onto~$A_\varepsilon^+$. Consequently, by virtue of~\eqref{Kreg} we obtain
\begin{align*}
\left| \int_{B_\rho(x) \setminus B_\varepsilon(x)} \nabla \psi(x) \cdot (y - x) K(x, y) \, dy \right| & \le |\nabla \psi(x)| \int_{A_\varepsilon^+} |y - x| \left| K(x, y) - K(x, 2 x - y) \right| dy \\
& \le \Gamma |\nabla \psi(x)| \int_{B_\rho(x)} |y - x|^{- n + \nu} \, dy \\
& = \frac{n |B_1| \Gamma |\nabla \psi(x)| \rho^{\nu}}{\nu},
\end{align*}
and thus~\eqref{I3}. Formula~\eqref{LKestsge} then follows from~\eqref{I1},~\eqref{I2} and~\eqref{I3}.
\end{proof}

Next, we state a measure theoretic result that assesses the size of the boundary of~\emph{dense} sets via a cubic grid decomposition. The proposition is based on the relative isoperimetric inequality and should probably be well-known to the experts in some equivalent form. However, as we have not been able to find a satisfactory reference in the literature, we present a proof of it in full details.

\begin{proposition} \label{densestAprop}
Let~$Q_r \subset \R^n$ be a closed cube of sides~$r > 0$ and~$A$ be an open subset of~$\R^n$. Suppose that there exists a constant~$c_\sharp \in (0, 1)$ such that
\begin{equation} \label{densestA}
\min \Big\{ \left| A \cap Q_r \right|, \left| Q_r \setminus \overline{A} \right| \Big\} \ge c_\sharp r^n.
\end{equation}
For any~$k \in \N$, let~$\PP$ be the non-overlapping (up to negligible sets) partition of~$Q_r$ in~$k^n$ closed cubes with sides of length~$r / k$, parallel to those of~$Q_r$. Then,
$$
\card \left( \Big\{ Q \in \PP : Q \cap \partial A \ne \varnothing \Big\} \right) \ge c_\star k^{n - 1},
$$
for some constant~$c_\star \in (0, 1)$ which depends only on~$n$ and~$c_\sharp$.
\end{proposition}
\begin{proof}
Define the following subclasses of~$\PP$:
\begin{align*}
\YY & := \left\{ Q \in \PP : Q \subseteq A \right\}, \\
\RR & := \left\{ Q \in \PP : Q \subseteq Q_r \setminus \overline{A} \right\}, \\
\GG & := \left\{ Q \in \PP : Q \cap \partial A \ne \varnothing \right\} = \PP \setminus \left( \YY \cup \RR \right).
\end{align*}
With this notation, we need to show that~$\card(\GG) \ge c_\star k^{n - 1}$.

To do this, we first divide~$\YY$ into connected clusters of adjacent cubes, i.e. we write
$$\YY = \bigcup_{j = 1}^{N_\YY} \YY_j,
$$
where each~$\YY_j$ is made up of adjacent cubes. Analogously, we write
$$
\RR = \bigcup_{j = 1}^{N_\RR} \RR_j.
$$
Of course,~$\card(\YY_j), \card(\RR_j) \le k^n$, for any~$j$. Moreover, we adopt the notation
$$
Y_j := \bigcup_{Q \in \YY_j} Q, \quad Y := \bigcup_{j = 1}^{N_\YY} Y_j, \quad R_j := \bigcup_{Q \in \RR_j} Q, \quad R := \bigcup_{j = 1}^{N_\RR} R_j \quad \mbox{and} \quad G := \bigcup_{Q \in \GG} Q,
$$
for the subsets of~$Q_r$ corresponding to the families~$\YY_j$,~$\YY$,~$\RR_j$,~$\RR$ and~$\GG$. In view of~\eqref{densestA},
$$
\Big( \card(\YY) + \card(\GG) \Big) \left( \frac{r}{k} \right)^n = \sum_{Q \in \YY \cup \GG} |Q| \ge |A \cap Q_r| \ge c_\sharp r^n.
$$
Hence,
$$
\mbox{either } \card(\YY) \ge \frac{c_\sharp}{2} k^n, \mbox{ or } \card(\GG) \ge \frac{c_\sharp}{2} k^n.
$$
With the same argument we obtain that
$$
\mbox{either } \card(\RR) \ge \frac{c_\sharp}{2} k^n, \mbox{ or } \card(\GG) \ge \frac{c_\sharp}{2} k^n.
$$
Now, if the bound for~$\card(\GG)$ is true, we are done. Thus, we assume the other two options to hold, so that
\begin{equation} \label{cardYj}
\sum_{j = 1}^{N_\YY} \card(\YY_j)^{(n - 1) / n} = \sum_{j = 1}^{N_\YY} \frac{\card(\YY_j)}{\card(\YY_j)^{1/n}} \ge \frac{1}{k} \sum_{j = 1}^{N_\YY} \card(\YY_j) = \frac{\card(\YY)}{k} \ge \frac{c_\sharp}{2} k^{n - 1},
\end{equation}
and analogously for the~$\RR_j$'s. But then, by the relative 
isoperimetric inequality in cubes (see e.g.~\cite[Corollary~5.9.13]{P12}),
\begin{equation} \label{perYj}
\mbox{Per} \left( Y_j, \accentset{\circ}{Q}_r \right) \ge c_1 \min \left\{ |Y_j|, |\accentset{\circ}{Q}_r \setminus Y_j| \right\}^{(n - 1) / n},
\end{equation}
for any~$j = 1, \ldots, N_\YY$ and for some dimensional constant~$c_1 \in (0, 1)$. Similarly for the~$R_j$'s. Notice now that
$$
\mbox{either } |Y_j| \le |\accentset{\circ}{Q}_r \setminus Y_j| \mbox{ for any } j = 1, \ldots, N_\YY, \mbox{ or } |R_j| \le |\accentset{\circ}{Q}_r \setminus R_j| \mbox{ for any } j = 1, \ldots, N_\RR.
$$
We assume, without loss of generality, that the above property holds for the~$Y_j$'s. By using~\eqref{cardYj} and~\eqref{perYj}, we get
\begin{align*}
\mbox{Per} \left( Y, \accentset{\circ}{Q}_r \right) & = \sum_{j = 1}^{N_\YY} \mbox{Per} \left( Y_j, \accentset{\circ}{Q}_r \right) \ge c_1 \sum_{j = 1}^{N_\YY} |Y_j|^{(n - 1) / n} \\
& = c_1 \left( \frac{r}{k} \right)^{n - 1} \sum_{j = 1}^{N_\YY} \card(\YY_j)^{(n - 1) / n} \ge c_2 r^{n - 1},
\end{align*}
for some~$c_2 \in (0, 1)$ depending only on~$n$ and~$c_\sharp$. But then, the~$Y_j$'s may only confine with the set~$G$. Hence,~$\mbox{Per}(G, \accentset{\circ}{Q}_r) \ge c_2 r^{n - 1}$ and thus
$$
2 n \left( \frac{r}{k} \right)^{n - 1} \card(\GG) = \sum_{Q \in \GG} \mbox{Per}(Q) \ge \mbox{Per}(G, \accentset{\circ}{Q}_r) \ge c_2 r^{n - 1},
$$
from which the thesis follows.
\end{proof}

The above result may be further sharpened as follows. In this enhanced form, such estimation will be used later in Section~\ref{ensec}.

\begin{corollary} \label{densestAcor}
Let~$Q_r \subset \R^n$ be a closed cube of sides~$r > 0$ and~$A$ be an open subset of~$\R^n$ for which~\eqref{densestA} holds true, for some~$c_\sharp \in (0, 1)$. Then, for any~$k \in \N$, there exists a collection~$\QQ = \{ Q^{(j)} \}_{j = 1}^N$ of non-overlapping closed cubes with sides of length~$r / k$, parallel to those of~$Q_r$, each~$Q^{(j)}$ contained in~$\accentset{\circ}{Q}_{2 r}$ and centered at some point
$$
x_j \in Q_r \cap \partial A.
$$
The cardinality~$N$ of the family can be chosen to satisfy
$$
N \ge c_{\star \star} k^{n - 1},
$$
for some constant~$c_{\star \star} \in (0, 1)$ depending only on~$n$ and~$c_\sharp$.
\end{corollary}
\begin{proof}
Given~$k \in \N$, divide~$Q_r$ into the non-overlapping partition~$\PP$ described in Proposition~\ref{densestAprop}. By virtue of Proposition~\ref{densestAprop} itself, we already know that the number of cubes in this partition that have non-trivial intersection with~$\partial A$ is at least~$c_\star k^{n - 1}$, for some constant~$c_\star \in (0, 1)$ which depends on~$n$ and~$c_\sharp$. Denote these cubes by~$\widetilde{Q}^{(j)}$, for~$j = 1, \ldots, \tilde{N}$, with~$\tilde{N} \ge c_\star k^{n - 1}$, and let~$x_j \in \widetilde{Q}^{(j)} \cap \partial A$ be the intersection points. We now translate each~$\widetilde{Q}^{(j)}$ to obtain a new cube~$Q^{(j)}$ centered at~$x_j$. Notice that, by reducing the total number of the cubes~$Q^{(j)}$'s to~$N := \tilde{N} / 3^n$, we may assume the new family to be non-overlapping. This concludes the proof.
\end{proof}

We conclude the section with the following simple lemma on the~$L^1$ convergence of the superlevel sets of a particular class of pointwise converging sequences of functions.

\begin{lemma} \label{diffsimmlimlem}
Let~$\Omega$ be an open bounded subset of~$\R^n$. Let~$\{ u_j \}$ be a sequence of measurable functions~$u_j: \Omega \to \R$ converging a.e.~in~$\Omega$ to~$u = \chi_E - \chi_{\Omega \setminus E}$, for some measurable set~$E \subseteq \Omega$. Then, given any~$\eta \in (-1, 1)$,
$$
\lim_{j \rightarrow +\infty} \left| E \Delta \{ u_j > \eta \} \right| = 0.
$$
\end{lemma}
\begin{proof}
First, note that
\begin{equation} \label{chitochi}
\chi_{ \{ u_j > \eta \} } \longrightarrow \chi_E \quad \mbox{a.e.~in } \Omega.
\end{equation}
Indeed, for a.e.~$x \in E$, we have~$u_j(x) \rightarrow u(x) = \chi_E(x) - \chi_{\Omega \setminus E}(x) = 1$. Hence,~$u_j(x) > \eta$ for any large enough~$j$, i.e.~$\chi_{ \{ u_j > \eta \} }(x) = 1 = \chi_E(x)$. Analogously, one checks that, for a.a.~$x \in \Omega \setminus E$, it holds~$\chi_{ \{ u_j > \eta \} } = 0 = \chi_E(x)$, for~$j$ sufficiently large. Then,~\eqref{chitochi} holds true.

To conclude the proof of the lemma, we simply apply formula~\eqref{chitochi} in combination with Lebesgue's dominated convergence theorem. We get
\begin{equation*}
\lim_{j \rightarrow +\infty} \left| E \Delta \{ u_j > \eta \} \right| = \lim_{j \rightarrow +\infty} \int_\Omega \left| \chi_E(x) - \chi_{ \{ u_j > \eta \}}(x) \right| \, dx = 0.\qedhere
\end{equation*}
\end{proof}

\section{Density estimates. Proof of Theorem~\ref{densestthm}} \label{densec}

In this section we establish the density estimates for the minimizers of~$\E$, thus proving Theorem~\ref{densestthm}. The argument follows the lines of that developed in~\cite{SV11,SV14}. Here we only sketch the general strategy and outline the main differences in our version.

As a first step, we construct the following barrier, that was also considered in~\cite{CP16}.

\begin{lemma} \label{barrierlem}
Let~$n \ge 1$,~$s \in (0, 1)$ and assume that~$K$ satisfies~\eqref{Ksymmetry} and~\eqref{Kbounds}. Given any~$\delta > 0$ there exists a constant~$C \ge 1$, depending on~$\delta$ and on universal quantities, such that for any~$C \le R$ we can construct a symmetric radially non-decreasing function
$$
w \in C^{1, 1}\left( \R^n, \left[ -1 + C^{-1} R^{- 2 s}, 1 \right] \right),
$$
with
$$
w = 1 \quad \mbox{in } \R^n \setminus B_R,
$$
which satisfies
\begin{equation} \label{LKwbar}
\left| \L_K w(x) \right| \le \delta \left( 1 + w(x) \right),
\end{equation}
and
\begin{equation} \label{1+wbarest}
\frac{1}{C} \left( R + 1 - |x| \right)^{- 2 s} \le 1 + w(x) \le C \left( R + 1 - |x| \right)^{- 2 s},
\end{equation}
for any~$x \in B_R$.
\end{lemma}

The proof of Lemma~\ref{barrierlem} is an adaptation of that of~\cite[Lemma~3.1]{SV14}
for the fractional Laplacian. In our case, the computations involved
are slightly more delicate, due to the more general form of the interaction kernel.

\begin{proof}[Proof of Lemma~\ref{barrierlem}]
Fix a value
\begin{equation} \label{r0def}
r_1 \ge 2^{3/s},
\end{equation}
and let~$r \ge r_1$. Then, set~$\ell(t) := (r - t)^{- 2 s}$, for any~$0 \le t < r$, and define
$$
\gamma_r := \left[ \ell(r - 1) - \ell(r / 2) - \ell'(r / 2) \left( r/2 - 1 \right) \right]^{-1}.
$$
Note that
\begin{align*}
\ell(r - 1) - \ell(r/2) - \ell'(r/2) (r/2 - 1) & = 1 - 2^{2 s} \left( 1 + 2 s - 4 s r^{- 1} \right) r^{- 2 s} \\
& \ge 1 - 12 {r_1}^{- 2 s} \\
& \ge 1 / 2,
\end{align*}
for any~$r \ge r_1$. Thus,~$\gamma_r$ is well-defined and
\begin{equation} \label{alpharbounds}
1 < \gamma_r \le 2.
\end{equation}

Consider the function~$h: [0, +\infty) \to [0, 1]$ defined by
$$
h(t) := \begin{cases}
0          & \mbox{if } t \in [0, r/2) \\
\gamma_r \left( \ell(t) - \ell(r/2) - \ell'(r/2) (t - r/2) \right) & \mbox{if } t \in [r/2, r - 1) \\
1          & \mbox{if } t \ge r - 1.
\end{cases}
$$
We have
$$
h(r/2) = 0, \quad h'(r/2) = 0 \quad \mbox{and} \quad h(r - 1) = 1,
$$
so that~$h \in C^{0, 1}([0, +\infty)) \cap C^{1, 1}([0, r - 1])$. Furthermore, recalling~\eqref{alpharbounds}, for~$t \in (r/2, r - 1)$ we have
\begin{equation} \label{h'h''}
\begin{aligned}
|h'(t)| & = \gamma_r |\ell'(t) - \ell'(r / 2)| = 2 s \gamma_r \left[ (r - t)^{- 2 s - 1} - (r / 2)^{- 2 s - 1} \right] \le 4 (r - t)^{- 2 s - 1} \\
|h''(t)| & = \gamma_r |\ell''(t)| = 2 s (2s + 1) \gamma_r (r - t)^{- 2 s - 2} \le 12 (r - t)^{- 2 s - 2}.
\end{aligned}
\end{equation}
As~$h$ is constant outside of~$(r/2, r - 1)$, the above estimates also holds for a.a.~$t \ge 0$.

We want to modify~$h$ between~$r - 2$ and~$r - 1$ in order to obtain a new function~$g$ of class~$C^{1, 1}$ on the whole half-line. To do this, let~$\eta \in C^\infty([0, + \infty))$ be a cut-off function with~$0 \le \eta \le 1$,~$\eta = 1$ in~$[0, r - 7 / 4]$,~$\eta = 0$ in~$[r - 5 / 4, +\infty)$,~$- 4 \le \eta' \le 0$ and~$|\eta''| \le 32$. We then set
$$
g(t) := \eta(t) h(t) + 1 - \eta(t) \quad \mbox{for any } t \ge 0.
$$
Of course,~$g \in C^{1, 1}([0, +\infty))$,~$0 \le g \le 1$ and~$g$ coincides with~$h$ outside of~$(r - 2, r - 1)$. On the other hand, by~\eqref{h'h''}, for~$t \in (r - 2, r - 1)$ we compute
\begin{align*}
|g'(t)| & \le |h'(t)| \chi_{[r - 2, r - 5/4]}(t) + 4 (1 - h(t)) \\
& \le 4 (r - t)^{- 2 s - 1} \chi_{[r - 2, r - 5/4]}(t) + 4 \\
& \le 40 \min \left\{ 1, (r - t)^{- 2 s - 1} \right\},
\end{align*}
and
\begin{align*}
|g''(t)| & \le |h''(t)| \chi_{[r - 2, r - 5/4]}(t) + 8 |h'(t)| \chi_{[r - 7/4, r - 5/4]}(t) + 32 (1 - h(t)) \\
& \le 12 (r - t)^{- 2 s - 2} \chi_{[r - 2, r - 5/4]}(t) + 32 (r - t)^{- 2 s -1} \chi_{[r - 7/4, r - 5/4]}(t) + 32 \\
& \le 600 \min \{ 1, (r - t)^{- 2 s - 2} \}.
\end{align*}
By combining these last two estimates with~\eqref{h'h''} (recall that~$g = h$ outside of~$(r - 2, r - 1)$), we conclude that there exists a numerical constant~$c_1 > 0$ such that
\begin{equation} \label{g'g''}
|g'(t)| \le c_1 \min\{ (r - t)^{- 2 s - 1}, 1 \} \quad \mbox{and} \quad |g''(t)| \le c_1 \min \{ (r - t)^{- 2 s - 2}, 1 \},
\end{equation}
for a.a.~$t \in [0, r]$. Moreover, we claim that
\begin{equation} \label{gbounds}
\min \left\{ (r - t)^{- 2 s}, 1 \right\} \le g(t) + 16 r^{- 2 s} \le 20 \min \left\{ (r - t)^{- 2 s}, 1 \right\} \quad \mbox{for any } t \in [0, r].
\end{equation}
Since the right-hand inequality of~\eqref{gbounds} follows almost directly from the definition of~$g$, we focus on the left estimation. The bound is clearly valid when~$t \ge r - 1$, as~$g = 1$ there. Also, when~$t \le r/2$, it holds~$g = 0$ and~$(r - t)^{- 2 s} \le 4 r^{- 2 s}$. Finally, when~$t \in (r/2, r - 1)$, using~\eqref{alpharbounds} we have
\begin{align*}
g(t) & \ge h(t) \ge \ell(t) - \ell(r/2) - \ell'(r/2) \left( t - r/2 \right) \\
& = (r - t)^{- 2 s} - (r/2)^{- 2 s} - 2 s (r/2)^{-2 s - 1} \left( t - r/2 \right) \\
& \ge (r - t)^{- 2 s} - 2^{2 s} \left( 1 +  2 s \right) r^{- 2 s} \\
& \ge (r - t)^{- 2 s} - 16 r^{- 2 s}.
\end{align*}
In any case,~\eqref{gbounds} is established.

Let now~$v(x) := g(|x|)$, for any~$x \in \R^n$. By the properties of~$g$, we recover that~$v \in C^{1, 1}(\R^n)$ is radially symmetric, radially non-decreasing and satisfies~$v = 0$ in~$B_{r/2}$,~$v = 1$ in~$\R^n \setminus B_r$. Moreover, we infer from~\eqref{gbounds} that, for~$x \in B_r$, it holds
\begin{equation} \label{vbounds}
\min \left\{ (r - |x|)^{- 2 s}, 1 \right\} \le v(x) + 16 r^{- 2 s} \le 20 \min \left\{ (r - |x|)^{- 2 s}, 1\right\}.
\end{equation}
We claim that for any~$x \in B_r$
\begin{equation} \label{supnablav1}
\| \nabla v \|_{L^\infty(B_{\max \left\{ (r - |x|) / 2, 1 \right\}}(x))} \le 
c_2 \max \left\{ \frac{r - |x|}{2}, 1 \right\}^{- 2 s - 1},
\end{equation}
and
\begin{equation} \label{supnablav2}
\| \nabla^2 v \|_{L^\infty (B_{\max \left\{ (r - |x|) / 2, 1 \right\}}(x))} \le c_2 \max \left\{ \frac{r - |x|}{2}, 1 \right\}^{- 2 s - 2},
\end{equation}
for some dimensional constant~$c_2 > 0$.  Estimate~\eqref{supnablav1} is almost immediate. Indeed, recalling~\eqref{g'g''}, for any~$y \in B_r$ we get
\begin{equation} \label{nablavy}
|\nabla v(y)| = |g'(|y|)| \le c_1 \min \{ (r - |y|)^{- 2 s - 1}, 1 \}.
\end{equation}
Fix now~$x \in B_r$ and take any~$y \in B_{(r - |x|) / 2}(x)$. Clearly,~$y \in B_r$. Also,
$$
|y| \le |y - x| + |x| \le \frac{r - |x|}{2} + |x| = \frac{r + |x|}{2},
$$
and thus
$$
r - |y| \ge r - \frac{r + |x|}{2} = \frac{r - |x|}{2}. 
$$
By this and~\eqref{nablavy}, it follows that
\begin{equation} \label{nablavy2}
\left| \nabla v(y) \right| \le c_1 \min \left\{ \left( \frac{r - |x|}{2} \right)^{- 2 s - 2}, 1 \right\},
\end{equation}
for any~$y \in B_{(r - |x|) / 2}(x)$. Finally, when~$(r - |x|) / 2 \le 1$ we use again~\eqref{nablavy} to deduce that~\eqref{nablavy2} holds also for~$y \in B_1(x) \cap B_r$. The proof of~\eqref{supnablav2} follows similarly, by noticing that, for~$y \in B_r \setminus \overline{B_{r/2}}$, by~\eqref{g'g''} it holds
\begin{align*}
\left| \nabla^2 v(y) \right| & \le |g''(|y|)| + \sqrt{2 (n + 1)} \, \frac{|g'(|y|)|}{|y|} \\
& \le 2 n c_1 \left( \min \left\{ \left( \frac{r - |y|}{2} \right)^{- 2 s - 2}, 1 \right\} + |y|^{-1} \min \left\{ \left( \frac{r - |y|}{2} \right)^{- 2 s - 1}, 1 \right\} \right) \\
& \le 50 n c_1 \min \{ (r - |y|)^{- 2 s - 2}, 1 \},
\end{align*}
where to obtain the last inequality we took advantage of the fact that~$|y| \ge r / 2 \ge r - |y|$ and~$r \ge r_1 > 2$, by~\eqref{r0def}.

With this in hand, we can deduce an estimate for~$L_{K_\sigma} v$ in~$B_r$, where~$K_\sigma$ is the scaled kernel defined by
$$
K_\sigma(x, y) := K(\sigma x, \sigma y), \quad \mbox{for a.a.~} x, y \in \R^n,
$$
with~$\sigma \ge 1$. Observe that~$K_\sigma$ satisfies~\eqref{Ksymmetry},~\eqref{Kbounds} and~\eqref{Kreg} with~$\lambda_\sigma := \sigma^{- n - 2 s} \lambda$,~$\Lambda_\sigma := \sigma^{- n - 2 s} \Lambda$ and~$\Gamma_\sigma := \sigma^{- n - 1 + \nu} \Gamma$. We apply either Lemma~\ref{LKlems<}, if~$s < 1/2$, or Lemma~\ref{LKlemsge}, if~$s \ge 1/2$, with~$\rho = \max \left\{ (r - |x|) / 2, 1 \right\}$. In view of~\eqref{supnablav1},~\eqref{supnablav2} and~\eqref{vbounds}, it is easy to see that, for any~$x \in B_r$,
\begin{equation} \label{LKv}
\left| \L_{K_\sigma} v(x) \right| \le c_3 \sigma^{- n - 1 + \bar{\nu}} \left( v(x) + 16 r^{-2 s} \right),
\end{equation}
for some constant~$c_3 > 0$, which may depend on~$n$,~$s$,~$\Lambda$ and also~$\Gamma$, if~$s \ge 1/2$, and where we set
$$
\bar{\nu} := \begin{cases}
1 - 2 s & \quad \mbox{if } s \in (0, 1/2) \\
\nu     & \quad \mbox{if } s \in [1/2, 1).
\end{cases}
$$
Without loss of generality, we may take
\begin{equation} \label{c3>delta}
c_3 \ge \delta.
\end{equation}

We are now able to construct the function~$w$. We take~$R \ge R_0$, with
\begin{equation} \label{R0def}
R_0 := \left( \frac{c_3}{\delta} \right)^{\frac{1}{1 - \bar{\nu}}} r_1,
\end{equation}
and set
\begin{equation} \label{rdef}
r := \frac{r_1}{R_0} R.
\end{equation}
Notice that~$r \ge r_1$. We then define
$$
w(x) := (2 - \beta) v\left( \frac{r}{R} x \right) + \beta - 1,
$$
where~$\beta := 32 r^{- 2 s}$. Clearly,~$\beta \in (0, 1)$, since~$r \ge r_1$ and~\eqref{r0def} is in force.

The function~$w$ thus obtained inherits all the qualitative properties of~$v$. That is,~$w$ is of class~$C^{1, 1}(\R^n)$, is radially symmetric and radially non-decreasing. Moreover,~$w = \beta - 1$ in~$B_{R / 2}$ and~$w = 1$ in~$\R^n \setminus B_R$. Now we check that~$w$ satisfies properties~\eqref{LKwbar} and~\eqref{1+wbarest}.

By changing variables appropriately, applying~\eqref{LKv} with~$\sigma := R / r \ge 1$, which holds thanks to~\eqref{c3>delta}, and recalling definitions~\eqref{R0def}-\eqref{rdef}, we get
\begin{align*}
\left| \L_K w(x) \right| & = (2 - \beta) \left( \frac{R}{r} \right)^n \left| \L_{K_{R/r}} v \left( \frac{r}{R} x \right) \right| \\
& \le c_3 (2 - \beta) \left( \frac{R}{r} \right)^{- 1 + \bar{\nu}} \left( v \left( \frac{r}{R} x \right) + 16 r^{- 2 s} \right) \\
& \le c_3 \left( \frac{R}{r} \right)^{- 1 + \bar{\nu}} \left(  (2 - \beta) v \left( \frac{r}{R} x \right) + 32 r^{- 2 s} \right) \\
& = \delta \left( 1 + w(x) \right),
\end{align*}
for any~$x \in B_R$. Thus,~\eqref{LKwbar} is established. The validity of the inequalities in~\eqref{1+wbarest} basically relies on~\eqref{vbounds}. Namely, being~$\beta$ positive and taking advantage of the upper estimate in~\eqref{vbounds}, along with~\eqref{R0def} and~\eqref{rdef}, we have, for any~$x \in B_R$, 
\begin{align*}
1 + w(x) & \le 2 \left[ v\left( \frac{r}{R} x \right) + 16 r^{- 2 s} \right] \\
& \le 40 \min \left\{ \left( \frac{c_3}{\delta} \right)^{\frac{2 s}{1 - \bar{\nu}}} \left( R - |x| \right)^{- 2 s}, 1 \right\} \\
& \le c_4 \left( R + 1 - |x| \right)^{- 2 s},
\end{align*}
for some constant~$c_4 > 0$ which depends on~$n, s, \Lambda, \Gamma, \nu$ and~$\delta$. The left-hand inequality of~\eqref{1+wbarest} follows similarly. Indeed, we first note that~$2 - \beta \ge 1$, since~$\beta \le 1$. Hence, by this,~\eqref{vbounds},~\eqref{R0def} and~\eqref{rdef}, we get
\begin{align*}
1 + w(x) \ge v\left( \frac{r}{R} x \right) + 16 r^{- 2 s} \ge \min \left\{ \left( \frac{c_3}{\delta} \right)^{\frac{2 s}{1 - \bar{\nu}}} \left( R - |x| \right)^{- 2 s}, 1 \right\} \ge c_5 \left( R + 1 - |x| \right)^{- 2 s},
\end{align*}
for some~$c_5 > 0$ which depends on the same parameters as~$c_4$.

The proof of the lemma is therefore complete, by eventually taking
\begin{equation*}
C := \max \left\{ R_0, \frac{1}{32} \left( \frac{r_1}{R_0} \right)^{2 s}, c_4, \frac{1}{c_5} \right\}. \qedhere
\end{equation*}
\end{proof}

We now proceed to the

\begin{proof}[Proof of Theorem~\ref{densestthm}]
Of course, it is enough to prove the result for~$x_0 = 0$. Also, we can restrict ourselves to show the validity of~\eqref{densest1} only, as the dual estimate~\eqref{densest2} can be then recovered in a completely analogous fashion or by just exchanging~$u$ with~$-u$.

Note that, thanks to a simple argument involving the energy estimate~\eqref{enest}, we can assume without loss of generality~$\theta \le \theta_0$. Furthermore, we are free to prove the existence of suitable constants~$\bar{c}$ and~$\bar{R}$ (as in the statement of the theorem) that \emph{both} depend on~$\theta$,~$\theta_0$, along with universal quantities. Then, employing again the energy estimates, we are able to prove that~$\bar{c}$ actually depends on universal quantities only. We refer to~\cite[Subsection~3.1]{SV14} for more details on this.

By the H\"{o}lder continuity of~$u$ (see e.g.~\cite[Section~2]{CV17}), we infer from~$u(0) \ge \theta_0$, which holds by hypothesis, that there exist two constants~$R_o \in (0, R)$ and~$\mu > 0$ such that
\begin{equation} \label{u>thetamu}
\left| \left\{ u > \theta \right\} \cap B_{R_o} \right| \ge \mu.
\end{equation}
Moreover, after a scaling argument, we are allowed to suppose~$\mu \ge e^n$. See again~\cite[Subsection~3.1]{SV14}.

Fix~$H > 2 (R_o + 1)$ and take~$R > 2 H$. Let~$w$ be the~$C^{1, 1}$ function constructed in Lemma~\ref{barrierlem}. Recall that
\begin{equation} \label{w=1prop}
w = 1 \quad \mbox{in } \R^n \setminus B_R.
\end{equation}
Define
\begin{equation} \label{vuw}
v := \min \left\{ u, w \right\},
\end{equation}
and observe that, by~\eqref{w=1prop},
\begin{equation} \label{v=uout}
v = u \quad \mbox{in } \R^n \setminus B_R.
\end{equation}
Writing for simplicity
$$
\K(u; B_R) := \frac{1}{2} \iint_{\C_{B_R}} |u(x) - u(y)|^2 K(x, y) \, dx dy,
$$
where~$\C_{B_R}$ is as in~\eqref{COmegadef}, in view of~\eqref{vuw},~\eqref{v=uout} and~\eqref{Ksymmetry} we have
\begin{align*}
& \K(u - v; B_R) + \K(v; B_R) - \K(u; B_R) \\
& \hspace{50pt} = - \iint_{\C_{B_R}} \left( (u - v)(x) - (u - v)(y) \right) \left( v(x) - v(y) \right) K(x, y) \, dx dy \\
& \hspace{50pt} = - \int_{\R^n} \int_{\R^n} \left( (u - v)(x) - (u - v)(y) \right) \left( v(x) - v(y) \right) K(x, y) \, dx dy \\
& \hspace{50pt} = - 2 \int_{\R^n} \int_{\R^n} \left( u(x) - v(x) \right) \left( v(x) - v(y) \right) K(x, y) \, dx dy \\
& \hspace{50pt} = - 2 \int_{B_R \cap \{ w = v < u \}} \left( u(x) - w(x) \right) \left( \PV \int_{\R^n} \left( w(x) - w(y) \right) K(x, y) \, dy \right) dx \\
& \hspace{50pt} \le 2 \int_{B_R \cap \{ w < u \}} \left( u(x) - w(x) \right) \left| \L_K w(x) \right| dx.
\end{align*}

The proof then continues as in~\cite{SV14}. We take advantage of the minimality of~$u$, together with hypotheses~\eqref{Wbound},~\eqref{W''} on~$W$ and estimates~\eqref{LKwbar},~\eqref{1+wbarest} provided by Lemma~\ref{barrierlem}
, to get
\begin{align*}
& \K(u - v; B_R) + \frac{\kappa}{2} \int_{B_R \cap \{ w < u \le \theta_\star \}} \left( 1 + w \right) \left( u - w \right) \, dx \\
& \hspace{90pt} \le c \int_{B_R \cap \{ u > \theta_\star \}} \left( R + 1 - |x| \right)^{- 2 s} \, dx - A(R),
\end{align*}
for some~$c > 0$, where~$\theta_\star$ is any fixed parameter that satisfies
\begin{equation} \label{thetastardef}
-1 < \theta_\star \le \min \left\{ \theta, -1 + \kappa \right\},
\end{equation}
and
$$
A(R) := \kappa \int_{B_R \cap \{ w < u \le \theta_\star \}} (u - w)^2 \, dx.
$$
Setting
$$
V(R) := \left| B_R \left\{ u > \theta_\star \right\} \right|,
$$
we use the coarea formula along with~\eqref{Kbounds} to obtain (recall that~$\xi \ge 3 R$)
\begin{equation} \label{densestthmtech1}
\begin{aligned}
& \frac{\lambda}{2} \int_{B_R} \int_{B_{2R}} \frac{|(u - v)(x) - (u - v)(y)|^2}{|x - y|^{n + 2 s}} \, dx dy + \frac{\kappa}{2} \int_{B_R \cap \{ w < u \le \theta_\star \}} \left( 1 + w \right) \left( u - w \right) \, dx \\
& \hspace{80pt} \le \K(u - v; B_R) + \frac{\kappa}{2} \int_{B_R \cap \{ w < u \le \theta_\star \}} \left( 1 + w \right) \left( u - w \right) \, dx \\
& \hspace{80pt} \le c \int_0^R \left( R + 1 - t \right)^{- 2 s} \left( \int_{\partial B_t} \chi_{\{ u > \theta_\star \}}(x) \, d\Haus^{n - 1}(x) \right) dt - A(R) \\
& \hspace{80pt} = c \int_0^R \left( R + 1 - t \right)^{- 2 s} V'(t) \, dt - A(R).
\end{aligned}
\end{equation}

Recalling the definition~\eqref{Psidef} of~$\Psi_s$, we claim that
\begin{equation} \label{VleV'}
V(R - H)^{\frac{n - 1}{n}} \Psi_s \left( \sqrt[n]{V(R - H} \right) \le \widetilde{C} \int_0^R (R + 1 - t)^{-2s} V'(t) \, dt,
\end{equation}
for some~$\widetilde{C} \ge 1$, provided~$H$ is large enough.

To prove~\eqref{VleV'}, we consider separately the two cases~$s < 1/2$ and~$s \ge 1/2$.

In the first situation we argue as in~\cite{SV11}. Recalling that~$u - v = 0$ outside of~$B_R$ and using the fractional Poincar\'e inequality, we compute
\begin{align*}
\int_{B_R} \int_{\R^n \setminus B_{2 R}} \frac{|(u - v)(x) - (u - v)(y)|^2}{|x - y|^{n + 2 s}} \, dx dy & = \int_{B_R} |(u - v)(x)|^2 \left( \int_{\R^n \setminus B_{2 R}} \frac{dy}{|x - y|^{n + 2 s}} \right) dx \\
& \le \frac{c}{R^{2 s}} \int_{B_{2 R}} |(u - v)(x)|^2 \, dx \\
& \le c \int_{B_R} \int_{B_{2 R}} \frac{|(u - v)(x) - (u - v)(y)|^2}{|x - y|^{n + 2 s}} \, dx dy,
\end{align*}
from which it follows immediately that
$$
\int_{B_R} \int_{B_{2R}} \frac{|(u - v)(x) - (u - v)(y)|^2}{|x - y|^{n + 2 s}} \, dx dy \ge c \int_{\R^n} \int_{\R^n} \frac{|(u - v)(x) - (u - v)(y)|^2}{|x - y|^{n + 2 s}} \, dx dy.
$$
By means of this and the fractional Sobolev inequality, we get
\begin{equation} \label{densestthmtech2}
\int_{B_R} \int_{B_{2R}} \frac{|(u - v)(x) - (u - v)(y)|^2}{|x - y|^{n + 2 s}} \, dx dy \ge c \| u - v \|_{L^{\frac{2 n}{n - 2 s}}(\R^n)}^2 = c \| u - v \|_{L^{\frac{2 n}{n - 2 s}}(B_R)}^2.
\end{equation}
Notice that, by taking~$H$ large enough (in dependence of~$\theta_\star$), by~\eqref{1+wbarest} we have
$$
w \le - 1 + \frac{1 + \theta_\star}{2} \quad \mbox{in } B_{R - H}.
$$
Hence,
\begin{equation} \label{u-vge}
|u - v| \ge u - v \ge u - w \ge \frac{1 + \theta_\star}{2} \quad \mbox{in } \{ u > \theta_\star \} \cap B_{R - H},
\end{equation}
so that
\begin{align*}
\| u - v \|_{L^{\frac{2 n}{n - 2 s}}(B_R)}^2 & \ge \left( \int_{\{ u > \theta_\star \} \cap B_{R - H}} |(u - v)(x)|^{\frac{2 n}{n - 2 s}} \, dx \right)^{\frac{n - 2 s}{n}} \\
& \ge \left( \frac{1 + \theta_\star}{2} \right)^2 \left| \{ u > \theta_\star \} \cap B_{R - H} \right|^{\frac{n - 2 s}{n}} \\
& = c \, V(R - H)^{\frac{n - 2 s}{n}}.
\end{align*}
This,~\eqref{densestthmtech1} and~\eqref{densestthmtech2} imply claim~\eqref{VleV'} for~$s < 1/2$.

On the other hand, when~$s \ge 1/2$ we follow the strategy displayed in~\cite{SV14}. Define
\begin{align*}
a_R & := \left\{ u - w \ge \frac{1 + \theta_\star}{4} \right\} \cap B_R,\\
b_R & := \left\{ \frac{1 + \theta_\star}{8} < u - w < \frac{1 + \theta_\star}{4} \right\} \cap B_R, \\
d_R & := \left( \R^n \setminus B_R \right) \cup \left( \left\{ u - w \le \frac{1 + \theta_\star}{8} \right\} \cap B_R \right).
\end{align*}
Note that, by~\eqref{u-vge}
$$
a_R \supseteq \left\{ u > \theta_\star \right\} \cap B_{R - H} \supseteq \left\{ u > \theta_\star \right\} \cap B_{R_o}.
$$
Hence, by~\eqref{u>thetamu} and~\eqref{thetastardef} we deduce that
\begin{equation} \label{aRgemu}
|a_R| \ge \mu,
\end{equation}
so that we may apply the geometric formula~(3.46) in~\cite{SV14} to obtain that
$$
\int_{a_R} \int_{d_R} \frac{dx dy}{|x - y|^{n + 2 s}} + |b_R| \ge 2 c_1 \, |a_R|^{\frac{n - 1}{n}} \Psi_s \left( \sqrt[n]{|a_R|} \right),
$$
for some constant~$c_1 > 0$. We then observe that
$$
\int_{a_R} \int_{\R^n \setminus B_{2 R}} \frac{dx dy}{|x - y|^{n + 2 s}} \le \frac{c_2}{R^{2 s}} |a_R|,
$$
for~$c_2 > 0$, and thus
\begin{align*}
\int_{a_R} \int_{d_R \cap B_{2 R}} \frac{dx dy}{|x - y|^{n + 2 s}} + |b_R| & = \int_{a_R} \int_{d_R} \frac{dx dy}{|x - y|^{n + 2 s}} + |b_R| - \int_{a_R} \int_{\R^n \setminus B_{2 R}} \frac{dx dy}{|x - y|^{n + 2 s}} \\
& \ge 2 c_1 \, |a_R|^{\frac{n - 1}{n}} \Psi_s \left( \sqrt[n]{|a_R|} \right) - \frac{c_2}{R^{2 s}} |a_R| \\
& = |a_R|^{\frac{n - 1}{n}} \Psi_s \left( \sqrt[n]{|a_R|} \right) \left[ 2 c_1 - \frac{c_2}{R^{2 s}} \frac{\sqrt[n]{|a_R|}}{\Psi_s \left( \sqrt[n]{|a_R|} \right)} \right] \\
& \ge |a_R|^{\frac{n - 1}{n}} \Psi_s \left( \sqrt[n]{|a_R|} \right) \left[ 2 c_1 - \frac{c_3}{R^{2 s - 1} \Psi_s \left( R \right)} \right],
\end{align*}
for some constant~$c_3 > 0$. Note that in the last line we took advantage of the fact that the function~$t / \Psi_s(t)$ is monotone increasing, at least\footnote{We point out that, by~\eqref{aRgemu} and the fact that we chose~$\mu \ge e^n$, it follows immediately that~$\sqrt[n]{|a_R|} \ge e$.} for~$t \ge e$. Since~$s \ge 1/2$, we may and do choose~$R$ sufficiently large for the quantity in square brackets to be larger than~$c_1$. This yields that
$$
\int_{a_R} \int_{d_R \cap B_{2 R}} \frac{dx dy}{|x - y|^{n + 2 s}} + |b_R| \ge c_1 |a_R|^{\frac{n - 1}{n}} \Psi_s \left( \sqrt[n]{|a_R|} \right).
$$
With the aid of this estimate and~\eqref{densestthmtech1}, an approach identical to that followed by~\cite{SV14} easily implies that~\eqref{VleV'} holds true also when~$s \ge 1/2$.

To conclude the proof of the theorem, we note that~\eqref{VleV'} is formula~(3.55) of~\cite{SV14}. An iterative procedure as the one performed there then finishes the argument.
\end{proof}

The density estimates that we just proved ensure that both
the sub- and superlevel sets of non-trivial minimizers have full measure in large balls. In the next result we use such estimates to deduce some information on the size of the interfaces of those minimizers.

\begin{proposition} \label{intdensprop}
Let~$u: \R^n \to [-1, 1]$ be a minimizer of~$\E$ in~$B_R(x_0)$, for some~$x_0 \in \R^n$ and~$R > 0$. Fix~$\theta, \theta_0 \in (0, 1)$ and suppose that~$u(x_0) \in [-\theta_0, \theta_0]$. Then, there exist two constants~$\tilde{c} \in (0, 1)$ and~$\tilde{R} \ge 4$ such that
\begin{equation} \label{intdensest}
\left| \left\{ |u| < \theta \right\} \cap B_R(x_0) \right| \ge \tilde{c} R^{n - 1},
\end{equation}
provided~$\tilde{R} \le R \le \xi$. The constants~$\tilde{c}$ and~$\tilde{R}$ depend only on~$\theta$,~$\theta_0$ and on universal quantities.
\end{proposition}
\begin{proof}
It is well-known that~$u$ is of class~$C^{0, \alpha}(B_{3 R / 4}(x_0))$, for some universal~$\alpha \in (0, 1)$, with H\"{o}lder norm bounded independently of~$R \ge 1$, that is
\begin{equation} \label{Calphaesttech}
\| u \|_{C^{0, \alpha}(B_{3 R / 4}(x_0))} \le C_1,
\end{equation}
for some universal constant~$C_1 \ge 1$. See e.g.~\cite{C17} or~\cite[Section~2]{CV17} for a thorough proof of this fact.

Now, we focus on the proof of~\eqref{intdensest}. Notice that we can restrict ourselves to take~$\theta > 0$ suitably small.

We initially assume that~$u(x_0) = 0$. We claim that there exists~$\delta \in (0, 1)$, depending on~$\theta$ and on universal quantities, for which
\begin{equation} \label{leveldist1}
\left\{ |u| < \theta \right\} \cap B_{R / 3}(x_0) \supseteq \left\{ |d| < \delta \right\} \cap B_{R / 3}(x_0),
\end{equation}
where we indicate with~$d$ the signed distance from the set~$\R^n \setminus \{ u > 0 \}$, that is
$$
d(x) := \begin{cases}
\dist \left( x, \R^n \setminus \{ u > 0 \} \right) & \quad \mbox{if } x \in \{ u > 0 \} \\
- \dist \left( x, \{ u > 0 \} \right) & \quad \mbox{if } x \in \R^n \setminus \{ u > 0 \},
\end{cases}
$$
for any~$x \in \R^n$. Note that
\begin{equation} \label{u>0d>0}
\left\{ d > 0 \right\} \cap B_{R/3}(x_0) = \left\{ u > 0 \right\} \cap B_{R/3}(x_0).
\end{equation}

To check that~\eqref{leveldist1} holds true, let~$\delta \in (0, 1)$ and, given~$x \in \{ |d| < \delta \} \cap B_{R / 3}(x_0)$, take~$y_x \in \{ u = 0 \} \cap B_{3 R / 4}(x_0)$ to be a point at which~$d(x) = |x - y_x|$. Then, by~\eqref{Calphaesttech} we have
$$
|u(x)| = |u(x) - u(y_x)| \le C_1 \left| x - y_x \right|^\alpha = C_1 d(x)^\alpha < C_1 \delta^\alpha,
$$
and hence~$|u(x)| < \theta$, if we choose~$\delta \le \left( \theta / C_1 \right)^{1 / \alpha}$. Thus,~\eqref{leveldist1} follows.

Also observe that, in particular, formulae~\eqref{leveldist1} and~\eqref{u>0d>0} yield that
\begin{equation} \label{leveldist2}
\left\{ u > \theta \right\} \cap B_{R / 3}(x_0) \subseteq \left\{ d > \delta \right\} \cap B_{R / 3}(x_0).
\end{equation}


After this preliminary work, we are in position to show that~\eqref{intdensest} is valid. We use~\eqref{leveldist1} and the coarea formula (applied to the function~$d$) to compute
\begin{equation} \label{leveldtech}
\begin{aligned}
\left| \left\{ |u| < \theta \right\} \cap B_{R / 3}(x_0) \right| & \ge \left| \left\{ |d| < \delta \right\} \cap B_{R / 3}(x_0) \right| = \int_{\left\{ |d| < \delta \right\} \cap B_{R / 3}(x_0)} |\nabla d(x)| \, dx \\
& = \int_{-\delta}^\delta \Per \left( \left\{ d = t \right\}, B_{R / 3}(x_0) \right) \, dt,
\end{aligned}
\end{equation}
as the gradient of the distance function has modulus equal to~$1$. Notice that~$\Per(E, \Omega)$ indicates the perimeter of a Borel set~$E$ inside a domain~$\Omega$. Now, we assume without loss of generality (up to changing~$u$ with~$-u$) that
$$
\left| \{ u > 0 \} \cap B_{R / 3}(x_0) \right| \le \left| B_{R / 3}(x_0) \setminus \{ u > 0 \} \right|,
$$
and we consider~$t \in [\delta / 2, \delta]$. Thanks to this reduction, the fact that~$\{ d > t \} \subset \{ d > 0\}$ and identity~\eqref{u>0d>0}, it is clear that
$$
\left| \{ d > t \} \cap B_{R / 3}(x_0) \right| \le \left| B_{R / 3}(x_0) \setminus \{ d > t \} \right|,
$$
so that the relative isoperimetric inequality yields
\begin{align*}
\Per \left( \{ d = t \}, B_{R / 3}(x_0) \right) & \ge c_2 
\min \left\{ \left| \{ d > t \} \cap B_{R / 3}(x_0) \right|, \left| B_{R / 3}(x_0) \setminus \{ d > t \} \right| \right\}^{(n - 1) / n} \\
& = c_2 \left| \{ d > t \} \cap B_{R / 3}(x_0) \right|^{(n - 1) / n} \\
& \ge c_2 \left| \{ d > \delta \} \cap B_{R / 3}(x_0) \right|^{(n - 1) / n},
\end{align*}
for some dimensional
constant~$c_2 > 0$. Hence, by taking advantage of this and of
estimate~\eqref{leveldist2}, in virtue of~\eqref{leveldtech} we get
\begin{align*}
\left| \left\{ |u| < \theta \right\} \cap B_{R/3}(x_0) \right| & \ge c_2 \int_{\delta/2}^\delta \left| \{ d > \delta \} \cap B_{R / 3}(x_0) \right|^{(n - 1) / n} dt \\
& \ge c_3 \left| \left\{ u > \theta \right\} \cap B_{R / 3}(x_0) \right|^{(n - 1) / n},
\end{align*}
for some constant~$c_3 > 0$ depending on~$\theta$ and on universal quantities. The thesis then follows by virtue of Theorem~\ref{densestthm}.

Now we only have to deal with the more general case of~$u(x_0) \in [-\theta_0, \theta_0]$. In this scenario, we know by Theorem~\ref{densestthm} that there exist two points~$x_1, x_2 \in B_{R / 3}(x_0)$ such that~$u(x_1) > 1/2$ and~$u(x_2) < - 1/2$. By continuity, we may thus find~$\bar{x} \in B_{R/3}(x_0)$ at which~$u$ vanishes. But then, we may apply the estimate obtained above to the ball~$B_{R / 3}(\bar{x}) \subset B_R(x_0)$ to conclude the proof.
\end{proof}

\section{Energy estimates. Proof of Theorem~\ref{enestbelowthm}} \label{ensec}

In this section, we show that the bound provided by Proposition~\ref{enestprop} is sharp, in the sense that the energy of every non-trivial (i.e. different from the constant functions~$1$ and~$-1$) minimizer can be also controlled from below by a term that has the same growth in~$R$ as the right-hand side of~\eqref{enest}. That is, we prove Theorem~\ref{enestbelowthm}.
\smallskip

The argument leading to~\eqref{enestbelow} changes significantly as~$s$ takes different values in the interval~$(0, 1)$. In particular, we first establish~\eqref{enestbelow} for the case~$s > 1/2$, by inspecting the potential term of~$\E$. Then, we look at the kinetic term~$\K$ to obtain the estimate when~$s < 1/2$. Finally, a deeper analysis of the contributions coming from~$\K$ yields the desired result also for~$s = 1/2$.

We stress that such differences in the proof of~\eqref{enestbelow} as~$s$ varies in~$(0, 1)$ 
are the effect of the competition between the local potential~$\P$ and the nonlocal interaction term~$\K$.
\smallskip

The precise result that we address here is slightly stronger than Theorem~\ref{enestbelowthm} and can be stated as follows. Note that throughout the section we always implicitly assume~$K$ and~$W$ to satisfy the hypotheses listed in Theorem~\ref{enestbelowthm}.

\begin{proposition} \label{enestbelowprop}
Let~$u: \R^n \to [-1. 1]$ be a minimizer for~$\E$ in~$B_R(x_0)$, for some~$x_0 \in \R^n$ and~$R > 0$. If~$u(x_0) \in [-\theta_0, \theta_0]$, for some~$\theta_0 \in (0, 1)$, then there exist two constants~$c_0 \in (0, 1)$ and~$R_0 \ge 1$, depending only on~$\theta_0$ and on universal quantities, such that
\begin{equation} \label{enestbelowbis}
\K(u; B_R(x_0), B_R(x_0)) + \P(u; B_R(x_0)) \ge c_0 R^{n - 1} \Psi_s(R),
\end{equation}
provided~$R_0 \le R \le \xi$.
\end{proposition}

Observe that~\eqref{enestbelowbis} provides a sharper estimate than~\eqref{enestbelow}, as the former does not involve the interaction term over~$B_R(x_0) \times (\R^n \setminus B_R(x_0) )$.
\smallskip

As anticipated at the beginning of the section, we remark that the method to prove~\eqref{enestbelowbis} is very sensitive
with respect to the fractional parameter~$s$. Indeed, when~$s\in(1/2,\,1)$, the proof of~\eqref{enestbelowbis} is considerably simpler than in the
other cases, and it follows essentially from Proposition~\ref{intdensprop}:
as a matter of fact, when~$s\in(1/2,\,1)$ the problem is ``sufficiently close to the local case'' that the optimal bounds are provided directly by the potential energy (which is, in a sense, of
local nature). When~$s\in (0,\,1/2]$ the situation is more complicated
and the kinetic energy plays a dominant role in the estimate.
In particular, a fine computation of the interactions
at all scales will be needed to detect the logarithmic
correction in the case~$s=1/2$.

In any case, it might be interesting to determine whether an estimate like~\eqref{enestbelowbis} holds true for both terms~$\K$ and~$\P$ separately.
\medskip

After this preliminary discussion, we now head to the proof of Proposition~\ref{enestbelowprop}. As a first result, we estimate from below the growth of the potential term.

\begin{lemma} \label{potestbelowlem}
Let~$u: \R^n \to [-1, 1]$ be a minimizer of~$\E$ in~$B_R(x_0)$, for some~$x_0 \in \R^n$ and~$R > 0$. If~$u(x_0) \in [-\theta_0, \theta_0]$, for some~$\theta_0 \in (0, 1)$, then there exist two constants~$c_1 \in (0, 1)$ and~$R_1 \ge 1$, depending on~$\theta_0$ and on universal quantities, such that
\begin{equation} \label{potestbelow}
\P(u; B_R(x_0)) \ge c_1 R^{n - 1},
\end{equation}
provided~$R_1 \le R \le \xi$.
\end{lemma}
\begin{proof}
Estimate~\eqref{potestbelow} is a simple consequence of Proposition~\ref{intdensprop} and hypothesis~\eqref{Wgamma}. Indeed,
$$
\P(u; B_R(x_0)) \ge \gamma(\theta_0) \left| \{ |u| < \theta_0 \} \cap B_R(x_0) \right| \ge c_1 R^{n - 1},
$$
for some~$c_1 > 0$ and provided~$R$ is large enough.
\end{proof}

From Lemma~\ref{potestbelowlem} we immediately deduce the validity of the bound~\eqref{enestbelowbis}, when~$s \in (1/2, 1)$. For others values of~$s$ we still get a bound from below for the total energy, which however is strictly weaker than the one claimed in Theorem~\ref{enestbelowthm}/Proposition~\ref{enestbelowprop}.

To cover the case of~$s \in (0, 1/2)$ we analyze the behavior of the interaction term~$\K$. A first computation in this direction is given by

\begin{lemma} \label{enestbelowlem}
Let~$u: B_r(x_0) \to \R$ be a measurable function, for some~$x_0 \in \R^n$ and~$r > 0$. Fix~$\theta_1 < \theta_2$ and let~$E, F \subset \R^n$ be two measurable sets on which~$u > \theta_2$ and~$u < \theta_1$, respectively. Suppose that
\begin{equation} \label{ABdens}
\min \left\{ \left| E \cap B_r(x_0) \right|, \left| F \cap B_r(x_0) \right| \right\} \ge c_\flat r^n,
\end{equation}
for some~$c_\flat > 0$. Then, there exists a constant~$c_* > 0$, depending on~$\theta_2 - \theta_1$,~$c_\flat$ and on universal quantities, for which
$$
\K(u; E \cap B_r(x_0), F \cap B_r(x_0)) \ge c_* r^{n - 2 s},
$$
provided~$2 r \le \xi$.
\end{lemma}
\begin{proof}
By taking advantage of~\eqref{Kbounds},~\eqref{ABdens} and the way~$E, F$ are chosen, we compute
\begin{align*}
\K(u; E \cap B_r(x_0), F \cap B_r(x_0)) & \ge \frac{\lambda}{2} \int_{E \cap B_r(x_0)} \int_{F \cap B_r(x_0)} \frac{|u(x) - u(y)|^2}{|x - y|^{n + 2 s}} \, dx dy \\
& \ge \frac{\lambda (\theta_2 - \theta_1)^2}{2 (2 r)^{n + 2 s}} \left| E \cap B_r(x_0) \right| \left| F \cap B_r(x_0) \right| \\
& \ge \frac{\lambda (\theta_2 - \theta_1)^2 c_\flat^2}{2^{n + 1 + 2 s}} \, r^{n - 2 s},
\end{align*}
if~$2 r \le \xi$. This concludes the proof.
\end{proof}

By combining this lemma with the density estimates of Theorem~\ref{densestthm}, we obtain the next corollary, which in particular establishes~\eqref{enestbelowbis} when~$s \in (0, 1/2)$.

\begin{corollary} \label{enestbelowcor2}
Let~$u: \R^n \to [-1, 1]$ be a minimizer of~$\E$ in~$B_R(x_0)$, for some~$x_0 \in \R^n$ and~$R > 0$. If~$u(x_0) \in [-\theta_0, \theta_0]$, for some~$\theta_0 \in (0, 1)$, then there exists two constants~$c_2 \in (0, 1)$ and~$R_2 \ge 3$, depending on~$\theta_0$ and on universal quantities, for which
\begin{equation} \label{enestbelow2}
\K(u; B_R(x_0), B_R(x_0)) \ge c_2 R^{n - 2 s},
\end{equation}
provided that~$R_2 \le R \le \xi$.
\end{corollary}
\begin{proof}
Apply Lemma~\ref{enestbelowlem} with, say,~$r = R/2$,~$\theta_1 = - 1/2$,~$\theta_2 = 1/2$,~$E = \{ u > 1/2 \}$,~$F = \{ u < - 1 / 2 \}$. Note that condition~\eqref{ABdens} is satisfied by virtue of Theorem~\ref{densestthm}, provided that we take~$R \ge 2 \max \{ \bar{R}(1/2, \theta_0), \bar{R}(- 1/2, \theta_0) \}$.
\end{proof}

In view of Lemma~\ref{potestbelowlem} and Corollary~\ref{enestbelowcor2}, we are only left to prove~\eqref{enestbelowbis} in the case~$s = 1/2$. This task is carried out in the following lemma, where we use the bound given in~\eqref{enestbelow2} repeatedly and at different scales, to gain the desired logarithmic correction (see also~\cite[Lemma~6.7]{CP16} for a similar computation in dimension~$n = 1$; of course, the case~$n\ge2$
that we consider here provides additional geometric and analytic difficulties
and a finer estimate is needed in order to detect the optimal contribution
of all the interactions).

\begin{lemma}
Assume~$s = 1/2$. Let~$u: \R^n \to [-1, 1]$ be a minimizer of~$\E$ in~$B_R(x_0)$, for some~$x_0 \in \R^n$ and~$R > 0$. If~$u(x_0) \in [-\theta_0, \theta_0]$, for some~$\theta_0 \in (0, 1)$, then there exists two constants~$c_3 \in (0, 1)$ and~$R_3 \ge 1$, depending on~$\theta_0$ and on universal quantities, such that
\begin{equation} \label{enestbelows=1/2}
\K(u; B_R(x_0), B_R(x_0)) \ge c_3 R^{n - 1} \log R,
\end{equation}
provided~$R_3 \le R \le \xi$.
\end{lemma}

\begin{proof}
Fix~$R_3 \ge \max \{ \bar{R}(0, \theta_0), \bar{R}(1/2, 0), \bar{R}(-1/2,0) \}$, with~$\bar{R}$ as given by Theorem~\ref{densestthm}, and take~$R \ge R_3$. Let~$k$ be the largest integer such that~$\sqrt{n} R_3 10^k \le R$. Notice that, for this choice, we have~$Q_{R_3 10^k}(x_0) \subset B_{R/2}(x_0)$ and
\begin{equation} \label{kNest}
R < \sqrt{n} R_3 10^{k + 1}.
\end{equation}
Note that we may suppose without loss of generality that~$k \ge 2$, since otherwise~\eqref{enestbelows=1/2} would immediately follow from Corollary~\ref{enestbelowcor2}.

Let~$\ell \in \{ 1, \ldots, k - 1 \}$. We claim that there exist two families of sets~$\{ E_i^{(\ell)} \}_{i = 1}^{N_\ell}, \{ F_i^{(\ell)} \}_{i = 1}^{N_\ell} \subset B_R(x_0)$, such that~$E_i^{(\ell)} \cap E_j^{(\ell)} = \emptyset$ for any~$i \ne j$,
\begin{equation} \label{EFdistclaim}
\dist \left( E_i^{(\ell)}, F_i^{(\ell)} \right) \ge \frac{R_3}{2} 10^\ell, \quad \sup_{x \in E_i^{(\ell)}, \, y \in F_i^{(\ell)}} |x - y| \le \frac{7 R_3}{2} 10^\ell,
\end{equation}
\begin{equation} \label{Kbelowestclaim}
\K(u; E_i^{(\ell)}, F_i^{(\ell)}) \ge c_\circ 10^{\ell (n - 1)},
\end{equation}
for any~$i$, and~$N_\ell \ge c_\circ 10^{(k - \ell) (n - 1)}$, for some constant~$c_\circ \in (0, 1)$ independent of~$k$,~$\ell$ and~$i$.

To prove the claim, we first use Theorem~\ref{densestthm} to deduce that
$$
\min\Big\{
| \{ u > 0 \} \cap Q_{R_3 10^k}(x_0) |, \;
| \{ u < 0 \} \cap Q_{R_3 10^k}(x_0) | \Big\}\ge c_\sharp 10^{n k},
$$
for some~$c_\sharp \in (0,1)$ independent of~$k$. Then, Corollary~\ref{densestAcor} yields the existence of a family~$\{ Q_i^{(\ell)} \}_{i = 1}^{N_\ell} \subset B_R(x_0)$ of~$N_\ell$ non-overlapping cubes with sides of length~$R_3 10^\ell$, each centered at a point at which~$u$ vanishes and such that~$N_\ell \ge c_{\star \star} 10^{(k - \ell) (n - 1)}$, for some~$c_{\star \star} \in (0,1)$ independent of~$k$ and~$\ell$.

For any~$i$, denote with~$B_i^{(\ell)} \subset Q_i^{(\ell)}$ the ball of radius~$R_3 10^\ell / 2$ concentric to~$Q_i^{(\ell)}$. Then, consider another ball~$\widetilde{B}_i^{(\ell)} \subset B_R(x_0)$ of radius~$R_3 10^\ell / 2$ and centered at any point at distance~$2 R_3 10^\ell$ from the center of~$B_i^{(\ell)}$. It holds that either
\begin{enumerate}[$(i)$]
\item $\widetilde{B}_i^{(\ell)} \subset \{ u > 0 \}$, or
\item $\widetilde{B}_i^{(\ell)} \subset \{ u < 0 \}$, or
\item $u(\widetilde{x}) = 0$ at some point~$\widetilde{x} \in \widetilde{B}_i^{(\ell)}$.
\end{enumerate}
In case~$(i)$ we set
$E_i^{(\ell)} := \{ u < -1/2 \} \cap B_i^{(\ell)}$, while in 
case~$(ii)$,~$E_i^{(\ell)} := \{ u > 1/2 \} \cap B_i^{(\ell)}$; in both 
cases we set~$F_i^{(\ell)} := \widetilde{B}_i^{(\ell)}$. On the other 
hand, in case~$(iii)$ we slightly translate~$\widetilde{B}_i^{(\ell)}$ 
(along a vector of length at most~$R_3 10^\ell / 2$) to make it centered 
at~$\widetilde{x}$ and set
\begin{eqnarray*}
&& E_i^{(\ell)} := \{ u < -1/2 \} \cap 
B_i^{(\ell)}\\
{\mbox{and }}&& F_i^{(\ell)} := \{ u > 1/2 \} \cap 
\widetilde{B}_i^{(\ell)}.\end{eqnarray*}
In any case, we employ Lemma~\ref{enestbelowlem} in combination with Theorem~\ref{densestthm} to see that~\eqref{Kbelowestclaim} is true. Inequalities~\eqref{EFdistclaim} also hold by construction.

The claim is therefore proved. In view of this,
$$
\left( E_i^{(\ell)} \times F_i^{(\ell)} \right) \cap \left( E_j^{(m)} \times F_j^{(m)} \right) = \emptyset,
$$
for any~$i, \in \{ 1, \ldots, N_\ell \}$,~$j \in \{ 1, \ldots, N_m \}$ and~$\ell, m \in \{ 1, \ldots, k - 1 \}$ such that~$(i, \ell) \ne (j, m)$. Accordingly, we can sum up each contribution coming from~\eqref{Kbelowestclaim} and obtain that
\begin{align*}
\K(u; B_R(x_0), B_R(x_0)) & \ge \sum_{\ell = 1}^{k - 1} \sum_{i = 1}^{N_\ell} \K(u; E_i^{(\ell)}, F_i^{(\ell)}) \ge c_\circ \sum_{\ell = 1}^{k - 1} N_\ell 10^{\ell(n - 1)} \\
& \ge c_\circ^2 \sum_{\ell = 1}^{k - 1} 10^{k(n - 1)} \ge \frac{c_\circ^2}{2} 10^{k(n - 1)} k.
\end{align*}
Recalling now~\eqref{kNest} and possibly enlarging~$R_3$, estimate~\eqref{enestbelows=1/2} plainly follows.
\end{proof}

The proof of Proposition~\ref{enestbelowprop} (and, consequently, of Theorem~\ref{enestbelowthm}) is therefore concluded.
\smallskip

As a consequence of Proposition~\ref{enestbelowprop}, we deduce a~\emph{clean ball condition} for the sub- and superlevel sets of the minimizers of~$\E$, that improves Theorem~\ref{densestthm}.

\begin{proposition} \label{cleanballprop}
Let~$u: \R^n \to [-1, 1]$ be a minimizer of~$\E$ in~$B_R(x_0)$, for some~$x_0 \in \R^n$ and~$R > 0$. Fix~$\theta, \theta_0 \in (0, 1)$ and a parameter~$\delta \in (0, 1]$. If~$u(x_0) \in [- \theta_0, \theta_0]$, then there exists two constants~$\kappa \in (0, 1]$,~$\widehat{R} > 0$ and two points~$z_1, z_2 \in \R^n$ such that
\begin{equation} \label{cleanball}
B_{\kappa R}(z_1) \subseteq \left\{ u > \theta \right\} \cap B_R(x_0)
\quad \mbox{and} \quad
B_{\kappa R}(z_2) \subseteq \left\{ u < -\theta \right\} \cap B_R(x_0),
\end{equation}
provided~$R \ge \widehat{R}$ and~$\delta \kappa R \le \xi$. The constants~$\kappa$ and~$\widehat{R}$ only depend on~$\delta$,~$\theta$,~$\theta_0$ and on universal quantities.
\end{proposition}

\begin{proof}
We restrict ourselves to the proof of the first inclusion in~\eqref{cleanball}, the other one being completely analogous.

Set~$\kappa := 1/(2 \sqrt{n} N)$, for some~$N \in \N$ to be later determined. Let~$Q$ be a cube centered at~$x_0$, of sides~$R / \sqrt{n}$. By construction,~$Q \subset B_{R / 2}(x_0)$. Subdivide~$Q$ into a family~$\{ Q_j \}_{j = 1}^{N^n}$ of non-overlapping cubes of sides~$R/(\sqrt{n} N) = 2 \kappa R$ parallel to those of~$Q$.

Let~$\QQ$ be the family of cubes~$Q_j$ having non-empty intersection with the level set~$\{ u > \theta \}$ and denote with~$\widetilde{N} \in \N \cup \{ 0 \}$ the cardinality of~$\QQ$. We relabel the cubes belonging to~$\QQ$ in such a way that
$$
\QQ = \left\{ \widetilde{Q}_j : j = 1, \ldots, \widetilde{N} \right\}.
$$
We claim that there exist two constants~$\widetilde{R} > 0$ and~$\tilde{c} \in (0, 1)$, depending on~$\theta$,~$\theta_0$ and on universal quantities, such that
\begin{equation} \label{tildeN}
\widetilde{N} \ge \tilde{c} N^n,
\end{equation}
provided
\begin{equation} \label{RNcond1}
R \ge \widetilde{R}.
\end{equation}
To see this, we consider the ball~$\widetilde{B} = B_{R / (2 \sqrt{n})}(x_0) \subset Q$ and we apply to it the density estimate~\eqref{densest1}. Setting~$\widetilde{R} := 2 \sqrt{n} \bar{R}$, with~$\bar{R} = \bar{R}(\theta, - \theta_0)$ as in Theorem~\ref{densestthm}, by~\eqref{RNcond1} we compute
$$
\bar{c} \left( \frac{R}{2 \sqrt{n}} \right)^n \le \left| \{ u > \theta \} \cap \widetilde{B} \right| = \sum_{j = 1}^{\widetilde{N}} \left| \widetilde{Q}_j \cap \{ u > \theta \} \cap \widetilde{B} \right| \le \sum_{j = 1}^{\widetilde{N}} \left| \widetilde{Q}_j \right| = \widetilde{N} \left( \frac{R}{\sqrt{n} N} \right)^n,
$$
which leads to~\eqref{tildeN}.

Now, either
\begin{enumerate}[$(i)$]
\item there exists~$j = 1, \ldots, \widetilde{N}$ such that~$\widetilde{Q}_j \cap \{ |u| \le \theta \} = \varnothing$, or
\item for any~$j = 1, \ldots, \widetilde{N}$, there exists~$y_j \in \widetilde{Q}_j$ at which~$|u(y_j)| \le \theta$.
\end{enumerate}
We claim that the latter possibility cannot occur, at least if~$N$ and~$R$ are sufficiently large, in dependence of~$\theta$,~$\theta_0$ and universal quantities only. If this is the case, condition~$(i)$ would then be valid. By the way in which the family~$\QQ$ is chosen and the continuity of~$u$, we might then conclude that the first assertion in~\eqref{cleanball} holds true.

By contradiction, suppose~$(ii)$ to be in force. By reducing, if necessary, the number~$\widetilde{N}$ of cubes in~$\QQ$ by a factor~$3^n$ and changing the position of the~$y_j$, we may assume without loss of generality that each~$\widetilde{Q}_j$ is centered at~$y_j$. Let now~$B_j^{(\delta)} \subset \widetilde{Q}_j$ be the ball of radius~$\delta R / (2 \sqrt{n} N)$ centered at~$y_j$. In view of~\eqref{tildeN} and Proposition~\ref{enestbelowprop}, we estimate
\begin{equation} \label{enestbelowtech}
\begin{aligned}
\E(u; B_R(x_0)) & \ge \sum_{j = 1}^{\widetilde{N}} \left[ \K(u; B_j^{(\delta)}, B_j^{(\delta)}) + \P(u; B_j^{(\delta)}) \right] \\
& \ge c_0 \widetilde{N} \left( \frac{\delta R}{2 \sqrt{n} N} \right)^{n - 1} \Psi_s \left( \frac{\delta R}{2 \sqrt{n} N} \right) \\
& \ge c_1 N \delta^{n - 1} R^{n - 1} \Psi_s \left( \frac{\delta R}{N} \right),
\end{aligned}
\end{equation}
for some~$c_0, c_1 \in (0, 1)$ depending only on~$\theta$ and on universal quantities. This is true provided
\begin{equation} \label{RNcond2}
R_* \le \frac{\delta R}{N} \le 2 \sqrt{n} \xi,
\end{equation}
with~$R_* \ge 1$ depending only on universal quantities and~$\theta$. On the other hand, with the aid of Proposition~\ref{enestprop} we estimate
\begin{equation} \label{enestabovetech}
\E(u; B_R(x_0)) \le C_2 R^{n - 1} \Psi_s(R),
\end{equation}
for some universal~$C_2 \ge 1$, if
\begin{equation} \label{RNcond3}
R \ge 3.
\end{equation}

By comparing the two estimates~\eqref{enestbelowtech} and~\eqref{enestabovetech}, we find that
\begin{equation} \label{RNcond4}
N \le \bar{C}_\delta := \begin{dcases}
C_\star^{\frac{1}{2 s}} \delta^{1 - \frac{n}{2s}} & \quad \mbox{if } s < 1/2, \\
2 C_\star \delta^{1 - n} & \quad \mbox{if } s = 1/2, \\
C_\star \delta^{1 - n} & \quad \mbox{if } s > 1/2,
\end{dcases}
\qquad \mbox{with } C_\star := \frac{C_2}{c_1},
\end{equation}
under the further restriction
\begin{equation} \label{RNcond5}
\frac{N}{\delta} \le \sqrt{R}, \quad \mbox{when } s = 1/2.
\end{equation}
But then we reached a contradiction, as we can choose~$N$ and~$\widehat{R}$ in such a way that conditions~\eqref{RNcond1},~\eqref{RNcond2},~\eqref{RNcond3} and~\eqref{RNcond5} are satisfied, but not~\eqref{RNcond4}. For instance, we may take
$$
N := \lfloor \bar{C}_\delta \rfloor + 2 \quad \mbox{and} \quad \widehat{R} := \left\{ 3, \widetilde{R}, N R_* / \delta, (N / \delta)^2 \right\},
$$
and any~$R \ge \widehat{R}$ such that~$\delta R / N \le 2 \sqrt{n} \xi$.
\end{proof}

\section{Planelike minimizers of~$\E$. Proof of Theorem~\ref{tauPLthm}} \label{PLminsec}

After having established in Proposition~\ref{cleanballprop} the appropriate clean ball condition for our problem, the proof of Theorem~\ref{tauPLthm} follows now the strategy exploited in~\cite{V04} and~\cite{CV17}. In particular, we shall follow the latter reference closely and only point out the most important differences in the argument.
\medskip

First, we restrict to consider a periodicity~$\tau$ larger than a big constant~$\tau_0$, chosen in dependence of~$\theta$ and universal quantities only. Note that this assumption does not harm the generality of our framework. Indeed, one can deal with a periodicity~$1 \le \tau \le \tau_0$ by scaling down and reducing it to the case of periodicity~$1$ treated in~\cite{CV17}. More specifically, we take~$\tau > \tau_0$, with~$\tau_0$ equal to twice the constant~$\widehat{R}$ of Proposition~\ref{cleanballprop} (applied with~$\delta = 1 / \sqrt{n}$ and~$\theta_0 = \theta$).

We only consider the case of a \emph{rational} direction~$\omega \in \tau \Q^n \setminus \{ 0 \}$ and a kernel~$K$ with rapid decay at infinity, i.e. that satisfies
$$
K(x, y) \le \frac{\Lambda_0}{|x - y|^{n + \beta}} \quad \mbox{for a.a.~} x, y \in \R^n \mbox{ such that } |x - y| \ge \rho_0, \mbox{ with } \beta > 1,
$$
for some~$\Lambda_0, \rho_0 > 0$. The general case can be then dealt with via approximation arguments analogous to those presented in Subsection~4.7 and Section~5 of~\cite{CV17}, respectively.

Under these additional assumptions, for~$M > 0$ we define the set of periodic, locally~$L^2$ functions
$$
L^2_\loc(\tRn) := \Big\{ u \in L^2_\loc(\R^n) : u \mbox{ is periodic with respect to~$\sim$} \Big\},
$$
the class of admissible functions
$$
\A := \left\{ u \in L^2_\loc(\tRn) : u(x) \ge \theta \mbox{ if } \omega \cdot x \le 0 \mbox{ and } u(x) \le - \theta \mbox{ if } \omega \cdot x \ge M \right\},
$$
the auxiliary functional
\begin{align*}
\F(u) := & \, \K(u; \tRn, \R^n) + \P(u; \tRn) \\
= & \, \frac{1}{2} \int_{\tRn} \int_{\R^n} |u(x) - u(y)|^2 K(x, y) \, dx dy + \int_{\tRn} W(x, u(x)) \, dx,
\end{align*}
and the set of absolute minimizers
$$
\M := \Big\{ u \in \A : \F(u) \le \F(v) \mbox{ for all } v \in \A \Big\}.
$$

We have

\begin{proposition} \label{uMexprop}
There exists a particular minimizer~$\u \in \M$ that satisfies the following properties:
\begin{enumerate}[$(i)$]
\item $\u$ is the~(unique)~\emph{minimal minimizer} of~$\M$, that is
$$
\u(x) = \inf_{u \in \M} u(x) \quad \mbox{for a.a.~} x \in \R^n,
$$
at least in the sense of~\cite[Definition~4.2.1]{CV17};
\item $\u$ is a minimizer of~$\E$ in every bounded open set~$\Omega$ compactly contained in the strip
$$
\S_\omega^M := \Big\{ x \in \R^n : \omega \cdot x \in [0, M] \Big\};
$$
\item $\u$ has the~\emph{doubling property}, i.e. it coincides with the minimal minimizers corresponding to the weaker equivalence relations~$\sim_m$, defined for any~$m \in \N^{n - 1}$ by setting
$$
x \sim_m y \quad \mbox{iff} \quad x - y \mbox{ belongs to the lattice generated by } m_1 \vec{z}_1, \ldots, m_{n - 1} \vec{z}_{n - 1},
$$
where~$\vec{z}_1, \ldots, \vec{z}_{n - 1} \in \tau \Z^n \setminus \{ 0 \}$ are vectors orthogonal to~$\omega$, forming a basis for the lattice induced by~$\sim$ (see~\cite[Subsection~4.3]{CV17} for more details);
\item the superlevel~[sublevel] sets of~$\u$ enjoy the~\emph{$\tau$-Birkhoff property} with respect to direction~$\omega$~[$-\omega$], that is
\begin{align*}
\left\{ \pm \u > \eta \right\} + k & \subseteq \left\{ \pm \u > \eta \right\} \mbox{ for any } k \in \tau \Z^n \mbox{ such that } \pm \omega \cdot k \le 0, \mbox{ and} \\
\left\{ \pm \u > \eta \right\} + k & \supseteq \left\{ \pm \u > \eta \right\} \mbox{ for any } k \in \tau \Z^n \mbox{ such that } \pm \omega \cdot k \ge 0,
\end{align*}
for any~$\eta \in \R$.
\end{enumerate}
\end{proposition}

Proposition~\ref{uMexprop} can be proved by repeating the arguments of~\cite[Subsections~4.1-4.5]{CV17}. The differences are purely formal and are due to the fact that here the equivalence relation~$\sim$ is chosen coherently with the~$\tau \Z^n$-periodicity of~$K$ and~$W$ (which were instead~$\Z^n$-periodic in~\cite{CV17}).

To conclude the proof of Theorem~\ref{tauPLthm}, we now need to show that for large, universal values of~$M / (|\omega| \tau)$, the minimal minimizer~$\u$ is~\emph{unconstrained}, i.e. that~$\u$ is a minimizer for~$\E$ in~\emph{any} bounded open subset~$\Omega$ of~$\R^n$ (compare this with point~$(ii)$ of Proposition~\ref{uMexprop}). The essential step in this direction is given by

\begin{proposition} \label{disttauprop}
There exists a constant~$M_0 > 0$, depending only on~$\theta$ and on universal quantities, such that if~$M \ge M_0 |\omega| \tau$, then the superlevel set~$\{ \u > - \theta \}$ is at least at distance~$\tau$ from the upper constraint~$\{ \omega \cdot x = M \}$ delimiting~$\S_\omega^M$.
\end{proposition}
\begin{proof}
First, we claim that
\begin{equation} \label{cleanballtauclaim}
\mbox{there exists a ball } B \mbox{ of radius } \sqrt{n} \tau \mbox{ contained in } \S_\omega^M, \mbox{ such that } |\u| > \theta \mbox{ on } B,
\end{equation}
if~$M \ge M_0 |\omega| \tau$, for some~$M_0 > 0$.

To prove claim~\eqref{cleanballtauclaim}, we fix~$\bar{x} \in \S_\omega^M$ in such a way that~$\omega \cdot \bar{x} = M / 2$. Consider the ball~$B_{\sqrt{n} \tau}(\bar{x})$, which is itself contained in~$\S_\omega^M$, provided~$M > 2 \sqrt{n} |\omega| \tau$. Now, either
\begin{enumerate}[$(i)$]
\item $|\u| > \theta$ on~$B_{\sqrt{n} \tau}(\bar{x})$, or
\item there exists~$x_0 \in B_{\sqrt{n} \tau}(\bar{x})$ at which~$|\u(x_0)| \le \theta$.
\end{enumerate}
In case~$(i)$, we clearly have a ball~$B$ as desired, and~\eqref{cleanballtauclaim} follows. Thus, we suppose that~$(ii)$ is valid.

Consider a large ball~$B_R(x_0)$, for~$R > 0$, and observe that~$B_R(x_0) \subset \subset \S_\omega^M$, if we take~$M \ge 4 (\sqrt{n} \tau + R) |\omega|$. Observe that, by point~$(ii)$ of Proposition~\ref{uMexprop},~$\u$ is a minimizer of~$\E$ in~$B_R(x_0)$. Therefore, we may apply Proposition~\ref{cleanballprop} (with~$\delta = 1 / \sqrt{n}$ and~$\theta_0 = \theta$) to find a ball~$\widetilde{B} \subset B_R(x_0)$ of radius~$\kappa R$,~$\kappa \in (0, 1]$, on which, say,~$\u < - \theta$, provided~$R \ge \widehat{R}$ and~$\kappa R \le \sqrt{n} \tau$.\footnote{Note that the resulting interval for~$R$ is not the empty set, in view of the assumptions we made on~$\tau$ at the beginning of the section.} By choosing exactly~$\kappa R = \sqrt{n} \tau$ we are led once again to~\eqref{cleanballtauclaim}, at least if~$M \ge 4 \sqrt{n} (1 + \kappa^{-1}) |\omega| \tau$.

Claim~\eqref{cleanballtauclaim} being proved, the rest of the proof follows in a more or less straightforward way. In view of the continuity of~$\u$ (see e.g.~\cite{C17} or~\cite[Section~2]{CV17}), we see that either~$\u > \theta$ on~$B$ or~$\u < - \theta$ on~$B$. By a general result on Birkhoff sets (an appropriately scaled version of~\cite[Proposition~4.5.2]{CV17}, for instance) and point~$(iv)$ of Proposition~\ref{uMexprop}, it then follows that either~$\u > \theta$ on a half-space of the form~$\{ \omega \cdot x < t_1 \}$, with~$t_1 > \sqrt{n} |\omega| \tau$, or that~$\u < - \theta$ on a half-space of the form~$\{ \omega \cdot x > t_2 \}$, with~$t_2 < M - \sqrt{n} |\omega| \tau$. As it is not hard to see, the former possibility leads to a contradiction,~$\u$ being the minimal minimizer. Consequently, the latter conclusion is true and the proposition is proved.
\end{proof}

Thanks to Proposition~\ref{disttauprop}, we see that~$\u$ starts attaining values smaller than~$-\theta$ well before meeting the upper constraint, once~$M$ is chosen sufficiently large (in relation to~$\tau$ and~$|\omega|$). By this fact and the minimality properties of~$\u$, it is not hard to see that the following stronger statement holds true.

\begin{proposition}
Let~$M \ge M_0 |\omega| \tau$, with~$M_0$ as in Proposition~\ref{disttauprop}. Then,~$\u = u_\omega^{M + a}$, for any~$a \ge 0$.
\end{proposition}

We refer to the arguments of~\cite[Corollary~4.6.2]{CV17} for a detailed proof of this result.

We remark that the above proposition shows in particular that~$\u$ is a minimizer of~$\E$ in each bounded subset of the half-space~$\{ \omega \cdot x > 0 \}$ (recall point~$(ii)$ of Proposition~\ref{uMexprop}). By exploiting another time the fact that~$\u$ is a minimal minimizer, it can be proved that its minimizing properties do not stop at the lower constraint~$\{ \omega \cdot x = 0 \}$, but in fact extends to the whole space~$\R^n$. We refer once again to the arguments displayed in~\cite[Subsection~4.6]{CV17} for more details on this.

Consequently,~$\u$ is a class~A minimizer and the proof of Theorem~\ref{tauPLthm} is concluded.

\section{The~$\Gamma$-convergence result. Proof of Proposition~\ref{Gammaconvprop}} \label{Gammasec}

The proof is an obvious modification of the argument contained in~\cite[Section~2]{SV12}. We report it here for the sake of completeness.
\medskip

We first establish the~$\Gamma$-liminf inequality claimed in~$(i)$. To this aim, note that we may assume
$$
\ell := \liminf_{\varepsilon \rightarrow 0^+} \E_\varepsilon(u_\varepsilon; \Omega) < +\infty.
$$
Let~$u_{\varepsilon_j}$ be a subsequence attaining the liminf above. Up to extracting a further subsequence, we may also assume that~$u_{\varepsilon_j}$ converges to~$u$ a.e.~in~$\R^n$. In view of this,
$$
\lim_{j \rightarrow +\infty} \frac{1}{\varepsilon_j^{2 s}} \int_{\Omega} W(x, u_{\varepsilon_j}(x)) \, dx \le \ell,
$$
that is, by the continuity of~$W$ and Fatou's lemma,
$$
\int_{\Omega} W(x, u(x)) \, dx = \lim_{j \rightarrow +\infty} \int_{\Omega} W(x, u_{\varepsilon_j}(x)) \, dx = 0.
$$
By~\eqref{Wgamma}, we then deduce that~$u$ only takes the values~$1$ and~$-1$ and, hence, may be written in~$\Omega$ as~$u = \chi_E - \chi_{\R^n \setminus E}$, for some measurable set~$E$. Accordingly, using again Fatou's lemma,
$$
\G(u; \Omega) = \K(u; \Omega) \le \lim_{j \rightarrow +\infty} \K(u_{\varepsilon_j}; \Omega) \le \ell,
$$
and the inequality in~$(i)$ holds true.

On the other hand, the proof of~$(ii)$ is almost immediate. We simply select~$u_\varepsilon = u$ as recovery sequence. Of course we may restrict ourselves to suppose~$\G(u; \Omega) < +\infty$ and consequently deduce that~$u = \chi_E - \chi_{\R^n \setminus E}$ in~$\Omega$, for some measurable set~$E$. As, by~\eqref{Wzeros},~$W(x, u(x)) = 0$ for a.a.~$x \in \Omega$, we get
$$
\G(u; \Omega) = \K(u; \Omega) = \E_\varepsilon(u; \Omega) = \E_\varepsilon(u_\varepsilon; \Omega), \quad \mbox{for any } \varepsilon > 0,
$$
and the thesis plainly follows.

The proof of Proposition~\ref{Gammaconvprop} is therefore concluded.

\section{Planelike minimal surfaces for~$\mbox{\normalfont Per}_K$. Proof of Theorem~\ref{PerPLthm}} \label{PerPLsec}

In this conclusive section, we obtain planelike class~A minimal surfaces for the~$K$-perimeter as limits in~$\varepsilon \rightarrow 0^+$ of the minimizers constructed in Theorem~\ref{epsPLthm}. That is, we prove Theorem~\ref{PerPLthm}.

Note that, in contrast with the previous sections, here we always consider~$s \in (0, 1/2)$.

For the sake of clarity, we also stress that, while~$K$ is obviously assumed throughout the section to satisfy the hypotheses explicitly stated in Theorem~\ref{PerPLthm}, the functional~$\E_\varepsilon$ that we often consider has to be intended as given by any~$W$ fulfilling requirements~\eqref{Wzeros}-\eqref{Wper}.
\medskip

As a first step towards the proof of Theorem~\ref{epsPLthm}, we need to have uniform-in-$\varepsilon$ versions of the density and energy estimates for the functional~$\E_\varepsilon$. Such results are reported in the next two propositions.

\begin{proposition} \label{epsenestprop}
Let~$\varepsilon \in (0, 1),$~$x_0 \in \R^n$ and~$R \ge 3 \varepsilon$. If~$u_\varepsilon: \R^n \to [-1, 1]$ is a minimizer of~$\E_\varepsilon$ in~$B_{R + 2 \varepsilon}(x_0)$, then
$$
\E_\varepsilon(u_\varepsilon; B_R(x_0)) \le C R^{n - 2 s},
$$
for some universal constant~$C \ge 1$.
\end{proposition}

\begin{proposition} \label{epsdensestprop}
Let~$\varepsilon \in (0, 1)$,~$x_0 \in \R^n$ and~$R > 0$. Let~$u_\varepsilon: \R^n \to [-1, 1]$ be a minimizer of~$\E_\varepsilon$ in~$B_R(x_0)$. Fix~$\theta, \theta_0 \in (-1, 1)$. If~$u_\varepsilon(x_0) \ge \theta_0$, then there exist two constants~$c \in (0, 1)$, depending on universal quantities, and~$\bar{R} > 0$, that may also depend on~$\theta$ and~$\theta_0$, such that
$$
\left| \left\{ u_\varepsilon > \theta \right\} \cap B_R(x_0) \right| \ge c R^n,
$$
provided~$\bar{R} \varepsilon \le R \le \xi/3$. Similarly, if~$u_\varepsilon(x_0) \le \theta_0$, then
$$
\left| \left\{ u_\varepsilon < \theta \right\} \cap B_R(x_0) \right| \ge c R^n,
$$
provided~$\bar{R} \varepsilon \le R \le \xi/3$.
\end{proposition}

Both statements can be easily deduced from Theorem~\ref{densestthm} and Proposition~\ref{enestprop}, respectively, by using~$\varepsilon$-rescalings of the form~\eqref{Reps}.
\smallskip

Next is a simple remark on the convergence of sequences having equibounded energies.

\begin{lemma} \label{directlem}
Let~$\Omega \subset \R^n$ be an open bounded set with Lipschitz boundary and~$\{ \varepsilon_j \}$ an infinitesimal sequence of positive numbers. Let~$\{ u_j \} \subset \X$ be a sequence of functions such that~$\E_{\varepsilon_j}(u_j; \Omega)$ is bounded uniformly in~$j$. Then, there exists a subsequence~$\{ j_k \}$ such that, as~$k \rightarrow +\infty$,
$$
u_{j_k} \longrightarrow u := \chi_{E} - \chi_{\R^n \setminus E} \quad \mbox{in } L^1(\Omega),
$$
for some measurable set~$E \subseteq \Omega$.
\end{lemma}
\begin{proof}
Recalling~\eqref{Kbounds} and the fact that~$|u_j| \le 1$, we estimate
\begin{align*}
[u_j]_{H^s(\Omega)}^2 & \le \int_{\Omega} \left[ \frac{1}{\lambda} \int_{B_\tau(x)} \left| u_j(x) - u_j(y) \right|^2 K(x, y) \, dy + 4 \int_{\R^n \setminus B_\tau(x)} \frac{dy}{|x - y|^{n + 2 s}} \right] dx  \\
& \le \frac{2}{\lambda} \, \E_{\varepsilon_j}(u_j; \Omega) + \frac{2 n |B_1| |\Omega|}{s \tau^{2 s}}.
\end{align*}
The equiboundedness of the energies allows us to apply e.g.~\cite[Corollary~7.2]{DPV12} and deduce the existence of a subsequence~$u_{j_k}$ converging to some~$u$ in~$L^1(\Omega)$. Also, by using again the equiboundedness of~$\{ \E_j(u_{\varepsilon_j}; \Omega) \}$ and Fatou's lemma, we have
$$
\int_{\Omega} W(x, u(x)) \, dx \le \liminf_{k \rightarrow +\infty} \int_{\Omega} W(x, u_{j_k}(x)) \, dx \le \liminf_{k \rightarrow +\infty} \varepsilon_{j_k}^{2 s} \, \E_{\varepsilon_{j_k}}(u_{j_k}; \Omega) = 0.
$$
By~\eqref{Wzeros}, the function~$u$ only takes values~$1$ and~$-1$, i.e.~$u = \chi_E - \chi_{\R^n \setminus E}$, for some~$E \subseteq \Omega$.
\end{proof}

The following result ensures that the limit of a sequence of class~A minimizers of~$\E_\varepsilon$ is a class~A minimizer of~$\G$. Note that this does not immediately follow from Proposition~\ref{Gammaconvprop} as a consequence of a standard application of the Fundamental Theorem of~$\Gamma$-convergence, since class~A minimizers are not \emph{global} minimizers on~$\X$.

\begin{lemma} \label{FTGClem}
Let~$\{ \varepsilon_j \}$ be an infinitesimal sequence of positive numbers. For any~$j \in \N$, let~$u_j \in \X$ be a class~A minimizer of~$\E_{\varepsilon_j}$ and suppose that the sequence~$\{ u_j \}$ converges in~$\X$ to~$u = \chi_E - \chi_{\R^n \setminus E}$, for some measurable set~$E \subseteq \R^n$, such that~$\G(u; \Omega) < +\infty$, for any bounded open~$\Omega \subset \R^n$. Then,~$u$ is a class~A minimizer of~$\G$ or, equivalently,~$\partial E$ is a class~A minimal surface for~$\Per_K$.
\end{lemma}
\begin{proof}
To prove the result it is enough to check that
$$
\G(u; B_R) \le \G(v; B_R) \quad \mbox{for any } v \in \X \mbox{ such that } v = u \mbox{ in } \R^n \setminus B_R \mbox{ and any } R > 0.
$$
Thus, fix~$R > 0$ and consider any such~$v$. Let~$\{ v_\varepsilon \}$ be the sequence converging to~$v$ in~$\X$ given by Proposition~\ref{Gammaconvprop}$(ii)$. As we also know that~$\{ u_j \}$ converges to~$u$ in~$\X$, Proposition~\ref{Gammaconvprop} yields
\begin{equation} \label{FTGCtech1}
\G(u; B_R) \le \liminf_{j \rightarrow +\infty} \E_{\varepsilon_j}(u_j; B_R) \quad \mbox{and} \quad \limsup_{j \rightarrow +\infty} \E_{\varepsilon_j}(v_{\varepsilon_j}; B_R) \le \G(v; B_R).
\end{equation}

For any~$j \in \N$, we set
$$
\bar{v}_j := \begin{cases}
v_{\varepsilon_j} & \quad \mbox{in } B_R \\
u_j & \quad \mbox{in } \R^n \setminus B_R.
\end{cases}
$$
Observe that~$\bar{v}_j = u_j$ outside of~$B_R$, so that, by the minimality of~$u_j$,
\begin{equation} \label{FTGCtech2}
\E_{\varepsilon_j}(u_j; B_R) \le \E_{\varepsilon_j}(\bar{v}_j; B_R).
\end{equation}

Define, for~$y \in \R^n \setminus B_R$,
\begin{equation} \label{Phidef}
\Phi(y) := \int_{B_R} \frac{dx}{|x - y|^{n + 2 s}}.
\end{equation}
Note that
\begin{equation} \label{PhiL1}
\Phi \in L^1(\R^n \setminus B_R).
\end{equation}
Indeed,
\begin{align*}
\int_{\R^n \setminus B_R} \left| \Phi(y) \right| dy & = \int_{B_R} \left[ \int_{\R^n \setminus B_R} \frac{dy}{|x - y|^{n + 2 s}} \right] dx \le \int_{B_R} \left[ \int_{\R^n \setminus B_{R - |x|}(x)} \frac{dy}{|x - y|^{n + 2 s}} \right] dx \\
& = \frac{n |B_1|}{2 s} \int_{B_R} \frac{dx}{\left( R - |x| \right)^{2 s}} \le \frac{n |B_1| R^{n - 1}}{2 s} \int_0^R \frac{d\rho}{\left( R - \rho \right)^{2 s}} = \frac{n |B_1| R^{n - 2 s}}{2 s (1 - 2 s)} \\
& < +\infty,
\end{align*}
as~$s < 1/2$. For~$x \in B_R$ and~$y \in \R^n \setminus B_R$, we then compute
\begin{align*}
\left| \left| \bar{v}_j(x) - \bar{v}_j(y) \right|^2 - \left| v_{\varepsilon_j}(x) - v_{\varepsilon_j}(y) \right|^2 \right| & = \left| \left| v_{\varepsilon_j}(x) - u_j(y) \right|^2 - \left| v_{\varepsilon_j}(x) - v_{\varepsilon_j}(y) \right|^2 \right| \\
& = \left| 2 v_{\varepsilon_j}(x) - u_j(y) - v_{\varepsilon_j}(y) \right| \left| v_{\varepsilon_j}(y) - u_j(y) \right| \\
& \le 4 \left| v_{\varepsilon_j}(y) - u_j(y) \right|.
\end{align*}
Hence, recalling the definition of~$\bar{v}_j$ and~\eqref{Kbounds} we get
\begin{equation} \label{FTGCtech3}
\begin{aligned}
& \lim_{j \rightarrow +\infty} \left| \E_{\varepsilon_j}(\bar{v}_j; B_R) - \E_{\varepsilon_j}(v_{\varepsilon_j}; B_R) \right| \\
& \hspace{30pt} = 2 \lim_{j \rightarrow +\infty} \left| \K(\bar{v}_j; B_R, \R^n \setminus B_R) - \K(v_{\varepsilon_j}; B_R, \R^n \setminus B_R) \right| \\
& \hspace{30pt} \le \lim_{j \rightarrow +\infty} \int_{B_R} \left[ \int_{\R^n \setminus B_R} \left| \left| \bar{v}_j(x) - \bar{v}_j(y) \right|^2 - \left| v_{\varepsilon_j}(x) - v_{\varepsilon_j}(y) \right|^2 \right| K(x, y) \, dy \right] dx \\
& \hspace{30pt} \le 4 \Lambda \lim_{j \rightarrow +\infty} \int_{\R^n \setminus B_R} \left| v_{\varepsilon_j}(y) - u_j(y) \right| \Phi(y) \, dy \\
& \hspace{30pt} = 0,
\end{aligned}
\end{equation}
where the last limit vanishes in view of Lebesgue's dominated convergence theorem, as both~$v_{\varepsilon_j}$ and~$u_j$ converge to~$u$ a.e.~in~$\R^n \setminus B_R$ and~\eqref{PhiL1} holds true.

By putting together~\eqref{FTGCtech1},~\eqref{FTGCtech2} and~\eqref{FTGCtech3}, we obtain
\begin{align*}
\G(u; B_R) & \le \liminf_{j \rightarrow +\infty} \E_{\varepsilon_j}(u_j; B_R) \le \liminf_{j \rightarrow +\infty} \E_{\varepsilon_j}(\bar{v}_j; B_R) \\
& \le \limsup_{j \rightarrow +\infty} \E_{\varepsilon_j}(v_{\varepsilon_j}; B_R) + \lim_{j \rightarrow +\infty} \left[ \E_{\varepsilon_j}(\bar{v}_j; B_R) - \E_{\varepsilon_j}(v_{\varepsilon_j}; B_R) \right] \\
& \le \G(v; B_R),
\end{align*}
and the thesis follows.
\end{proof}

Analogously, the limit of class~A minimal surfaces for the~$K$-perimeter is itself a class~A minimal surface.

\begin{lemma} \label{Perstablem}
For any~$k \in \N$, let~$\partial E_k$ be a class~A minimal surface for~$\Per_K$ and suppose that the sequence~$\{ E_k \}$ converges in~$\X$ to a set~$E \subset \R^n$, such that~$\Per_K(E, \Omega) < +\infty$, for any bounded open~$\Omega \subset \R^n$. Then,~$\partial E$ is a class~A minimal surface for~$\Per_K$.
\end{lemma}
\begin{proof}
The proof of this result is quite similar to that of Lemma~\ref{FTGClem}.

By possibly passing to a subsequence, we may suppose that~$\chi_{E_k} \rightarrow \chi_E$ a.e.~in~$\R^n$.

Given any~$R > 0$, we need to prove that
\begin{equation} \label{Perstabthesis}
\Per_K(E; B_R) \le \Per_K(F; B_R) \quad \mbox{for any set } F \mbox{ such that } F \setminus B_R = E \setminus B_R.
\end{equation}
Given such a~$F$, we set~$\bar{F}_k := \left( F \cap B_R \right) \cup \left( E_k \setminus B_R \right)$. Since~$\bar{F}_k \setminus B_R = E_k \setminus B_R$, we know that
\begin{equation} \label{Perstabtech1}
\Per_K(E_k; B_R) \le \Per_K(\bar{F}_k; B_R).
\end{equation}
Also, thanks to representation~\eqref{PerKchi} and Fatou's lemma, it easily follows that
\begin{equation} \label{Perstabtech2}
\Per_K(E; B_R) \le \liminf_{k \rightarrow +\infty} \Per_K(E_k; B_R).
\end{equation}

Observe now that
\begin{align*}
\Per_K(\bar{F}_k; B_R) - \Per_k(F; B_R) & = \LL_K(F \cap B_R, \R^n \setminus (E_k \cup B_R)) - \LL_K(F \cap B_R, \R^n \setminus (E \cup B_R)) \\
& \quad + \LL_K(E_k \setminus B_R, B_R \setminus F) - \LL_K(E \setminus B_R, B_R \setminus F).
\end{align*}
As
$$
\left| \LL_K(F \cap B_R, \R^n \setminus (E_k \cup B_R)) - \LL_K(F \cap B_R, \R^n \setminus (E \cup B_R)) \right| \le \LL_K((E_k \Delta E) \setminus B_R, F \cap B_R),
$$
and
$$
\left| \LL_K(E_k \setminus B_R, B_R \setminus F) - \LL_K(E \setminus B_R, B_R \setminus F) \right| \le \LL_K((E_k \Delta E) \setminus B_R, B_R \setminus F),
$$
we may compute
$$
\left| \Per_K(\bar{F}_k; B_R) - \Per_k(F; B_R) \right| \le \LL_K((E_k \Delta E) \setminus B_R, B_R) = \int_{\R^n \setminus B_R} \chi_{E_k \Delta E}(y) \Phi(y) \, dy,
$$
where~$\Phi \in L^1(\R^n \setminus B_R)$ is the function defined in~\eqref{Phidef}. Since~$\chi_{E_k \Delta E} \rightarrow 0$ a.e.~in~$\R^n$, using Lebesgue's dominated convergence theorem we deduce that
$$
\lim_{k \rightarrow +\infty} \Per_K(\bar{F}_k; B_R) = \Per_k(F; B_R).
$$
Consequently, by this,~\eqref{Perstabtech1} and~\eqref{Perstabtech2}, we conclude that
\begin{align*}
\Per_K(E; B_R) \le \liminf_{k \rightarrow +\infty} \Per_K(E_k; B_R) \le \lim_{k \rightarrow +\infty} \Per_K(\bar{F}_k; B_R) \le \Per_K(F; B_R),
\end{align*}
that is~\eqref{Perstabthesis}.
\end{proof}

With all these preliminary results at hand, we may now prove the main proposition of the section. Observe that, as a consequence of it, we deduce the validity Theorem~\ref{PerPLthm}, at least for the case of~$\omega \in \tau \Q^n \setminus \{ 0 \}$.

\begin{proposition} \label{omegaratprop}
Let~$\theta \in (0, 1)$ and~$\omega \in \R^n \setminus \{ 0 \}$ be fixed. For~$\varepsilon > 0$, let~$u_\varepsilon$ be the class~A minimizer of~$\E_{\varepsilon}$ associated to~$\theta$ and~$\omega$, constructed in Theorem~\ref{epsPLthm}. Then, there exists an infinitesimal sequence~$\{ \varepsilon_j \}$ of positive numbers such that
\begin{enumerate}[$(i)$]
\item $u_{\varepsilon_j}$ converges in~$\X$ to a function~$u = \chi_E - \chi_{\R^n \setminus E}$, for some measurable set~$E \subset \R^n$;
\item for any~$\eta \in (0, 1)$, the set~$\{ |u_{\varepsilon_j}| \le \eta \}$ converges locally uniformly to~$\partial E$;
\item $\partial E$ is a class~A minimal surface for~$\Per_K$ and, for any bounded set~$\Omega \subset \R^n$,
$$
\Per_K(E; \Omega) \le C_\Omega,
$$
where~$C_\Omega > 0$ is a constant depending only on~$\Omega$ and universal quantities;
\item there exists a universal constant~$c \in (0, 1)$ such that, given any point~$x_0 \in \partial E$,
$$
\left| E \cap B_R(x_0) \right| \ge c R^n \quad \mbox{and} \quad \left| B_R(x_0) \setminus E \right| \ge c R^n,
$$
for any~$0 < R < \xi$;
\item the inclusions
$$
\bigg\{ x \in \R^n : \frac{\omega}{|\omega|} \cdot x < 0 \bigg\} \subset E \subset \bigg\{ x \in \R^n : \frac{\omega}{|\omega|} \cdot x \le \tau M_0 \bigg\},
$$
hold true, where~$M_0 > 0$ is the constant found in Theorem~\ref{epsPLthm};
\item
if~$\omega \in \tau \Q^n \setminus \{ 0 \}$, then~$\partial E$ is periodic with respect to~$\sim_{\tau, \, \omega}$.
\end{enumerate}
\end{proposition}
\begin{proof}
Let~$\{ R_k \}$ be an increasing sequence of positive numbers that diverges to~$+\infty$, and~$\{ \varepsilon_j \}$ an infinitesimal sequence of positive numbers. Thanks to Proposition~\ref{epsenestprop}, for any~$k \in \N$ there exists a constant~$C_k > 0$ such that
$$
\E_{\varepsilon_j}(u_{\varepsilon_j}; B_{R_k}) \le C_k \quad \mbox{for any } j \in \N.
$$
Then, after a standard diagonal argument and repeated applications of Lemma~\ref{directlem}, it is easy to see that a subsequence of~$\{ u_{\varepsilon_j} \}$ (that we label in the same way) converges in~$\X$ and pointwise a.e.~in~$\R^n$ to a function~$u = \chi_{E} - \chi_{\R^n \setminus E}$, for some measurable set~$E \subseteq \R^n$. Moreover, given any bounded set~$\Omega \subset \R^n$, we may select a sufficiently large~$k \in \N$ so that~$\Omega \subseteq B_{R_k}$ and hence by Proposition~\ref{Gammaconvprop}$(i)$,
$$
4 \Per_K(E; \Omega) = \G(u; \Omega) \le \G(u; B_{R_k}) \le \liminf_{j \rightarrow +\infty} \E_{\varepsilon_j}(u_{\varepsilon_j}; B_{R_k}) \le C_k < +\infty.
$$
Consequently, Lemma~\ref{FTGClem} ensures that~$\partial E$ is a class~A minimal surface for~$\Per_K$. Also, when~$\omega \in \tau \Q^n \setminus \{ 0 \}$, then~$\partial E$ clearly inherits the periodicity properties shared by each element of the approximating sequence~$\{ u_{\varepsilon_j}\}$. We have therefore showed that~$(i)$,~$(iii)$ and~$(vi)$ hold true.

Concerning~$(v)$, observe that, by Theorem~\ref{epsPLthm}, we may assume without loss of generality that~$u_{\varepsilon_j} \ge \theta$ on~$\{ \omega \cdot x \le 0 \}$. Accordingly,
$$
u(x) = \lim_{j \rightarrow +\infty} u_{\varepsilon_j}(x) \ge \theta > 0 \quad \mbox{for a.a.~} x \mbox{ such that } \omega \cdot x \le 0.
$$
As~$u$ only attains the values~$1$ and~$-1$, we conclude that, up to changing~$u$ on a negligible set, it holds~$u = 1$ on~$\{ \omega \cdot x < 0 \}$. Similarly, one shows that~$u = -1$ on~$\{ \omega \cdot x \ge \tau M_0|\omega| \}$ and~$(v)$ readily follows.

Now we deal with the proof of~$(ii)$. We argue by contradiction and suppose that there exist a compact set~$K \subset \R^n$, a value~$\delta \in (0, \xi)$ and a sequence of points~$\{ x_j \} \subset K$, such that~$|u_{\varepsilon_j}(x_j)| \le \eta$ and~$B_\delta(x_j) \subset E$. Up to a subsequence,~$\{x_j\}$ converges to some point~$x_0 \in K$, with~$B_{\delta/2}(x_0) \subset E$. That is,
\begin{equation} \label{u=1delta/2}
u = 1 \mbox{ on } B_{\delta/2}(x_0).
\end{equation}
By Proposition~\ref{epsdensestprop}, for any large enough~$j$ we have
$$
\left| \{ u_{\varepsilon_j} < -1/2 \} \cap B_{\delta/2}(x_0) \right| \ge \left| \{ u_{\varepsilon_j} < -1/2 \} \cap B_{\delta / 4}(x_j) \right| \ge c \left| B_{\delta/2} \right|,
$$
for some~$c \in (0, 1)$ independent of~$j$. Accordingly,
\begin{align*}
\int_{B_{\delta/2}(x_0)} u_{\varepsilon_j}(x) \, dx & = \int_{\{ u_{\varepsilon_j} < -1/2 \} \cap B_{\delta/2}(x_0)} u_{\varepsilon_j}(x) \, dx + \int_{\{ u_{\varepsilon_j} \ge -1/2 \} \cap B_{\delta/2}(x_0)} u_{\varepsilon_j}(x) \, dx \\
& \le - \frac{1}{2} \left| \{ u_{\varepsilon_j} < -1/2 \} \cap B_{\delta/2}(x_0) \right| + \left| \{ u_{\varepsilon_j} \ge -1/2 \} \cap B_{\delta/2}(x_0) \right| \\
& \le \left( 1 - \frac{c}{2} \right) \left| B_{\delta/2} \right|.
\end{align*}
But then, taking advantage of this,~\eqref{u=1delta/2} and the fact that~$u_{\varepsilon_j} \to u$ in~$L^1(B_{\delta/2}(x_0))$, we get
\begin{align*}
\left| B_{\delta/2} \right| = \int_{B_{\delta / 2}(x_0)} u(x) \, dx & = \lim_{j \rightarrow +\infty} \int_{B_{\delta/2}(x_0)} u_{\varepsilon_j}(x) \, dx \le \left( 1 - \frac{c}{2} \right) |B_{\delta/2}|,
\end{align*}
which is a contradiction. Hence,~$(ii)$ holds true.

Finally, we show the validity of the density estimates stated in~$(iv)$. By the uniform convergence result of item~$(ii)$, we infer the existence of a sequence of points~$\{ x_j \} \subset B_{R/2}(x_0)$ at which~$|u_{\varepsilon_j}(x_j)| \le 1/2$. Proposition~\ref{epsdensestprop} then ensures that
$$
\left| \{ u_{\varepsilon_j} > 0 \} \cap B_R(x_0) \right| \ge \left| \{ u_{\varepsilon_j} > 0 \} \cap B_{R/3}(x_j) \right| \ge c R^n,
$$
for some universal constant~$c \in (0, 1)$. As, by point~$(i)$,~$u_{\varepsilon_j} \rightarrow u$ a.e.~in~$\R^n$, making use of Lemma~\ref{diffsimmlimlem} we obtain
$$
\left| E \cap B_{R}(x_0) \right| = \lim_{j \rightarrow +\infty} \left| \{ u_{\varepsilon_j} > 0 \} \cap {B_R(x_0)} \right| \ge c R^n.
$$
Similarly, one checks the validity of the estimate for the complement of~$E$.
%
\end{proof}

To complete the proof of Theorem~\ref{PerPLthm} we now only need to deal with directions~$\omega \in \R^n \setminus \tau \Q^n$. This is done in the following proposition, via an approximation argument.

\begin{proposition}
Let~$\omega \in \R^n \setminus \tau \Q^n$. Then, there exists a class~A minimal surface~$\partial E$ for~$\Per_K$, such that
\begin{equation} \label{Eplanelike}
\bigg\{ x \in \R^n : \frac{\omega}{|\omega|} \cdot x < 0 \bigg\} \subset E \subset \bigg\{ x \in \R^n : \frac{\omega}{|\omega|} \cdot x \le \tau M_0 \bigg\},
\end{equation}
Moreover, there exists a sequence~$\{ \omega_k \} \subset \tau \Q^n \setminus \{ 0 \}$ such that~$\omega_k \rightarrow \omega$ and, denoting by~$\partial E_k$ the class~A minimal surface for~$\Per_K$ associated with~$\omega_k$ given by Proposition~\ref{omegaratprop}, it holds~$E_k \rightarrow E$ in~$L^1_\loc(\R^n)$ and~$\partial E_k \rightarrow \partial E$ locally uniformly in~$\R^n$.
\end{proposition}
\begin{proof}
As a preliminary observation, notice that, if~$u = \chi_A - \chi_{\R^n \setminus A}$, for some measurable~$A \subseteq \R^n$, then
$$
\Per_K(A; \Omega) = \frac{1}{4} \, \E_\varepsilon(u; \Omega),
$$
for any~$\varepsilon > 0$ and any bounded set~$\Omega \subset \R^n$.

Let now~$\{ \omega_k \} \subset \tau \Q^n \setminus \{ 0 \}$ be any sequence converging to~$\omega$ and denote with~$\partial E_k$ the corresponding class~A minimal surface constructed in Proposition~\ref{omegaratprop}. Let~$\{ R_i \}$ be any monotone sequence of positive numbers, diverging to~$+\infty$ and~$\{ \varepsilon_k \}$ be any infinitesimal sequence. Then, by Proposition~\ref{omegaratprop}$(iii)$,
$$
\E_{\varepsilon_k}(\chi_{E_k} - \chi_{\R^n \setminus E_k}, B_{R_i}) = 4 \Per_K(E_k; B_{R_i}) \le C_i,
$$
for some constant~$C_i > 0$ that only depends on~$i$ and universal quantities. Consequently, with the aid of Lemma~\ref{directlem} and a diagonal argument analogous to that presented in the proof of Proposition~\ref{omegaratprop}, we obtain that, up to a subsequence,~$E_k$ converges in~$L^1_\loc(\R^n)$ and pointwise a.e.~in~$\R^n$ to some measurable~$E \subseteq \R^n$. The inclusions in~\eqref{Eplanelike} then readily follow from the analogous ones obtained in Proposition~\ref{omegaratprop}$(iv)$ for each~$E_k$. Moreover, using Fatou's lemma it is immediate to check that the~$K$-perimeter of~$E$ in any compact set is finite. Hence, by Lemma~\ref{Perstablem}, the set~$\partial E$ is a class~A minimal surface for~$\Per_K$.

We are therefore only left to show the locally uniform convergence of~$\partial E_k$ to~$\partial E$. The argument is similar to the one adopted in the proof of Proposition~\ref{omegaratprop}$(ii)$. Suppose by contradiction that there exist a compact set~$K$, a number~$\delta \in (0, \xi)$ and a sequence of points~$\{ x_k \}$ such that~$x_k \in \partial E_k \cap K$ and~$B_\delta(x_k) \cap E = \varnothing$. Up to a subsequence, we see that~$x_k \rightarrow x_0$, for some~$x_0 \in K$, and
\begin{equation} \label{EoutBdelta}
B_{\delta/2}(x_0) \cap E = \varnothing.
\end{equation}
In view of Proposition~\ref{omegaratprop}$(iv)$, there exists a universal constant~$c \in (0,1)$ for which
$$
\left| E_k \cap B_{\delta/2}(x_0) \right| \ge \left| E_k \cap B_{\delta / 4}(x_k) \right| \ge c,
$$
for any~$k \in \N$ sufficiently large. By the~$L^1_\loc$ convergence of the~$E_k$'s, we then have that
$$
\left| E \cap B_{\delta/2}(x_0) \right| = \lim_{k \rightarrow +\infty} \left| E_k \cap B_{\delta/2}(x_0) \right| \ge \left| E_k \cap B_{\delta / 4}(x_k) \right| \ge c,
$$
in contradiction with~\eqref{EoutBdelta}. The proof of the proposition is thus complete.
\end{proof}


\begin{thebibliography}{MBRS16$\,$}

\bibitem[AB98]{AB98}
G. Alberti, G. Bellettini,
\emph{A non-local anisotropic model for phase transitions: asymptotic behaviour of rescaled energies},
European J. Appl. Math., 9.3:261--284, 1998.

\bibitem[ABS94]{ABS94}
G. Alberti, G. Bouchitt{\'e}, P. Seppecher,
\emph{Un r\'esultat de perturbations singuli\`eres avec la norme {$H^{1/2}$}},
C. R. Acad. Sci. Paris S\'er. I Math., 319.4:333--338, 1994.

\bibitem[BV08]{BV08}
I. Birindelli, E. Valdinoci,
\emph{The {G}inzburg-{L}andau equation in the {H}eisenberg group},
Commun. Contemp. Math., 10.5:671--719, 2008.

\bibitem[BV17]{BV17}
C. Bucur, E. Valdinoci,
\emph{Nonlocal diffusion and applications},
Lecture Notes of the Unione Matematica Italiana, 20,
Springer, Unione Matematica Italiana, Bologna, xii+155, 2016. 

\bibitem[CC10]{CC10}
X. Cabr\'e, E. Cinti,
\emph{Energy estimates and 1-{D} symmetry for nonlinear equations involving the half-{L}aplacian},
Discrete Contin. Dyn. Syst., 28.3:1179--1206, 2010.

\bibitem[CC14]{CC14}
X. Cabr{\'e}, E. Cinti,
\emph{Sharp energy estimates for nonlinear fractional diffusion equations},
Calc. Var. Partial Differential Equations, 49.1-2:233--269, 2014.

\bibitem[CS14]{CS14}
X. Cabr\'e, Y. Sire,
\emph{Nonlinear equations for fractional Laplacians, I: Regularity, maximum principles, and Hamiltonian estimates},
Ann. Inst. H. Poincar\'e Anal. Non Lin\'eaire, 31.1:23--53, 2014.

\bibitem[CS15]{CS15}
X. Cabr\'e, Y. Sire,
\emph{Nonlinear equations for fractional Laplacians II: existence, uniqueness, and qualitative properties of solutions},
Trans. Amer. Math. Soc., 367.2:911--941, 2015.

\bibitem[CS-M05]{CS-M05}
X. Cabr\'e, J. Sol\`a-Morales,
\emph{Layer solutions in a half-space for boundary reactions},
Comm. Pure Appl. Math., 58.12:1678--1732, 2005.

\bibitem[CC95]{CC95}
L. Caffarelli, A. C{\'o}rdoba,
\emph{Uniform convergence of a singular perturbation problem},
Comm. Pure Appl. Math., 48.1:1--12, 1995.

\bibitem[CdlL01]{CdlL01}
L. Caffarelli, R. de la Llave,
\emph{Planelike minimizers in periodic media},
Comm. Pure Appl. Math., 54.12:1403--1441, 2001.

\bibitem[CRS10]{CRS10}
L. Caffarelli, J.-M. Roquejoffre, O. Savin,
\emph{Nonlocal minimal surfaces},
Comm. Pure Appl. Math., 63.9:1111--1144, 2010.

\bibitem[CGS84]{CGS84}
J. Carr, M. E. Gurtin, M. Slemrod,
\emph{Structured phase transitions on a finite interval},
Arch. Rational Mech. Anal., 86.4:317--351, 1984.

\bibitem[CSV16]{CSV16}
E. Cinti, J. Serra, E. Valdinoci,
\emph{Quantitative flatness results and $BV$-estimates for stable nonlocal minimal surfaces},
arXiv preprint, arXiv:1602.00540, 2016. In print on
J. Differential Geom.

\bibitem[C17]{C17}
M. Cozzi,
\emph{Regularity results and Harnack inequalities for minimizers and solutions of nonlocal problems: a unified approach via fractional De~Giorgi classes},
J. Funct. Anal., 272.11:4762--4837, 2017.

\bibitem[CDV17a]{CDV17a}
M. Cozzi, S. Dipierro, E. Valdinoci,
\emph{Nonlocal phase transitions in homogeneous and periodic media},
J. Fixed Point Theory Appl. 19.1:387--405, 2017.

\bibitem[CDV17b]{CDV17b}
M. Cozzi, S. Dipierro, E. Valdinoci,
\emph{Planelike interfaces in long-range Ising models and connections with
nonlocal minimal surfaces}, J. Stat. Phys. 167.6:1401--1451, 2017.

\bibitem[CP16]{CP16}
M. Cozzi, T. Passalacqua,
\emph{One-dimensional solutions of non-local {A}llen-{C}ahn-type equations with rough kernels},
J. Differential Equations, 260.8:6638--6696, 2016.

\bibitem[CV17]{CV17}
M. Cozzi, E. Valdinoci,
\emph{Plane-like minimizers for a non-local {G}inzburg-{L}andau-type energy in a periodic medium},
J. \'Ec. polytech. Math. 4:337--388, 2017.

\bibitem[D13]{D13}
G. D{\'a}vila,
\emph{Plane-like minimizers for an area-{D}irichlet integral},
Arch. Ration. Mech. Anal., 207.3:753--774, 2013.

\bibitem[DPV12]{DPV12}
E. Di Nezza, G. Palatucci, E. Valdinoci,
\emph{Hitchhiker's guide to the fractional {S}obolev spaces},
Bull. Sci. Math., 136:521--573, 2012.

\bibitem[DMV17]{DMV17}
S. Dipierro, M. Medina, E. Valdinoci,
\emph{Fractional elliptic problems with critical growth in the
whole of $\Bbb{R}^n$}, Appunti,
Scuola Normale Superiore di Pisa (Nuova Serie)
[Lecture Notes. Scuola Normale Superiore di Pisa (New Series)], 15,
Edizioni della Normale, Pisa, viii+152, 2017.

\bibitem[G17]{G17}
N. Garofalo, \emph{Fractional thoughts},
arXiv preprint,
arXiv:1712.03347, 2017.

\bibitem[G09]{G09}
M. d. M. Gonz{\'a}lez,
\emph{Gamma convergence of an energy functional related to the fractional {L}aplacian},
Calc. Var. Partial Differential Equations, 36.2:173--210, 2009.

\bibitem[G87]{G87}
M. E. Gurtin,
\emph{Some results and conjectures in the gradient theory of phase transitions},
Metastability and incompletely posed problems ({M}inneapolis, {M}inn., 1985), IMA Vol. Math. Appl., 3:135--146, 1987.

\bibitem[H32]{H32}
G. A. Hedlund,
\emph{Geodesics on a two-dimensional {R}iemannian manifold with periodic coefficients},
Ann. of Math. (2), 33.4:719--739, 1932.

\bibitem[L14]{L14}
M. Ludwig,
\emph{Anisotropic fractional perimeters},
J. Differential Geom., 96.1:77--93, 2014.

\bibitem[MSW16]{MSW16}
V. Millot, Y. Sire, K. Wang
\emph{Asymptotics for the fractional Allen-Cahn equation
and stationary nonlocal minimal surfaces},
arXiv preprint, arXiv:1610.07194, 2016.

\bibitem[M87]{M87}
L. Modica,
\emph{The gradient theory of phase transitions and the minimal interface criterion},
Arch. Rational Mech. Anal., 98.2:123--142, 1987.

\bibitem[MBRS16]{MBRS16}
G. Molica Bisci, V. D. Radulescu, R. Servadei,
\emph{Variational methods for nonlocal fractional problems},
Encyclopedia of Mathematics and its Applications, 162,
Cambridge University Press, Cambridge, xvi+383, 2016.

\bibitem[M24]{M24}
H. M. Morse,
\emph{A fundamental class of geodesics on any closed surface of genus greater than one},
Trans. Amer. Math. Soc., 26.1:25--60, 1924.

\bibitem[NV07]{NV07}
M. Novaga, E. Valdinoci,
\emph{The geometry of mesoscopic phase transition interfaces},
Discrete Contin. Dyn. Syst., 19.4:777--798, 2007.

\bibitem[PSV13]{PSV13} G. Palatucci, O. Savin, E. Valdinoci,
\emph{Local and global minimizers for a variational energy involving a fractional norm},
Ann. Mat. Pura Appl. (4), 192.4:673--718, 2013.

\bibitem[PV05]{PV05}
A. Petrosyan, E. Valdinoci,
\emph{Geometric properties of {B}ernoulli-type minimizers},
Interfaces Free Bound., 7.1:55--77, 2005.

\bibitem[PV05b]{PV05b}
A. Petrosyan, E. Valdinoci,
\emph{Density estimates for a degenerate/singular phase-transition model},
SIAM J. Math. Anal., 36.4:1057--1079(electronic), 2005.

\bibitem[P12]{P12}
W. F. Pfeffer,
\emph{The divergence theorem and sets of finite perimeter},
Pure and Applied Mathematics, CRC Press, Boca Raton, FL, xvi+243, 2012.

\bibitem[SV11]{SV11}
O. Savin, E. Valdinoci,
\emph{Density estimates for a nonlocal variational model via the {S}obolev inequality},
SIAM J. Math. Anal., 43.6:2675--2687, 2011.

\bibitem[SV12]{SV12}
O. Savin, E. Valdinoci,
\emph{{$\Gamma$}-convergence for nonlocal phase transitions},
Ann. Inst. H. Poincar\'e Anal. Non Lin\'eaire, 29.4:479--500, 2012.

\bibitem[SV14]{SV14}
O. Savin, E. Valdinoci,
\emph{Density estimates for a variational model driven by the {G}agliardo norm},
J. Math. Pures Appl. (9), 101.1:1--26, 2014.

\bibitem[SiV09]{SiV09}
Y. Sire, E. Valdinoci,
\emph{Fractional Laplacian phase transitions and boundary reactions: a geometric inequality and a symmetry result},
J. Funct. Anal., 256.6:1842--1864, 2009.

\bibitem[SiV12]{SiV12}
Y. Sire, E. Valdinoci,
\emph{Density estimates for phase transitions with a trace},
Interfaces Free Bound., 14.2:153--165, 2012.

\bibitem[V04]{V04}
E. Valdinoci,
\emph{Plane-like minimizers in periodic media: jet flows and {G}inzburg-{L}andau-type functionals},
J. Reine Angew. Math., 574:147--185 , 2004.

\end{thebibliography}
\end{document}